\pdfoutput=1

\documentclass[11pt]{amsart}
\usepackage[foot]{amsaddr}
\input{header.def}
\usepackage[backend=biber,style=zbsalph,backref=true]{biblatex}
\defbibheading{bibliography}[\refname]{\section*{#1}}
\renewcommand\bibfont\small

\newcommand{\viz}{\bm{z}}

\DeclareMathOperator{\Span}{span}
\DeclareMathOperator{\diag}{diag}
\newcommand{\polar}[1]{{#1}^{\circ}}
\DeclareMathOperator{\ch}{char}

\newcommand{\lf}{\mathcal{L}}

\title{Quasi-linear relation between partition and analytic rank}
\author[Guy Moshkovitz]{Guy Moshkovitz\textsuperscript{1}}
\address{\textsuperscript{1}Department of Mathematics, City University of New York (Baruch College \& Graduate Center), New York, NY 10010, USA}
\email{guymoshkov@gmail.com}
\author[Daniel G. Zhu]{Daniel G. Zhu\textsuperscript{2}}
\address{\textsuperscript{2}Department of Mathematics, Princeton University, Princeton, NJ 08544, USA}
\email{zhd@princeton.edu}

\thanks{The work on this paper began as part of the 2022 NYC Discrete Math REU, funded by NSF grant DMS 2051026.
G.M. is supported by NSF Award DMS-2302988.}
\begin{document}
\begin{abstract}
An important conjecture in additive combinatorics, number theory, and algebraic geometry posits 
that the partition rank and analytic rank of tensors are equal up to a constant, over any finite field.
We prove the conjecture up to a logarithmic factor.

Our proof is largely independent of previous work, utilizing recursively constructed polynomial identities and random walks on zero sets of polynomials. We also introduce a new, vector-valued notion of tensor rank (``local rank''), which serves as a bridge between partition and analytic rank, and which may be of independent interest as a tool for analyzing higher-degree polynomials.
\end{abstract}

\maketitle

\section{Introduction}

A fundamental goal of arithmetic geometry is to count the number of solutions to systems of polynomial equations over finite fields. In this area, one interesting direction is to understand systems of polynomials where the number of solutions deviates measurably from that of a random system, i.e.\ tuples of polynomials $P=(P_1,\ldots,P_m)$ in $n$ variables over a finite field $\F$ and of degree at most $k$, for which $\abs{\Pr_x[P(x)=0] - \abs{\F}^{-m}} \geq \delta$ for some $\delta$.\footnote{Throughout, all probabilistic expressions such as $\Pr_x$ refer to a uniformly random choice of $x$.} In this case, 
one hopes to upper-bound the \emph{rank}\footnote{Also known as the \emph{Schmidt rank}, \emph{$h$-invariant}, or \emph{strength}.} $r$, which, in the case where all the $P_i$ are homogeneous of degree $k \geq 2$, is defined as the smallest $r$ such that some nontrivial linear combination
of the $P_i$ can be written as the sum of $r$ reducible homogeneous polynomials.
Towards this goal, 
Davenport and Lewis~\cite{DavenportLe62} showed that in degree $k=3$ we have $\delta = O_n(\abs{\F}^{-\Omega(r)})$, 
which was later extended by Schmidt~\cite{Schmidt84,Schmidt85} to $\delta = O_{n,k}(\abs{\F}^{-\Omega_k(r)})$ for general $k$, assuming $\ch \F > k$. 
However, for these estimates to be useful, the number of variables must be bounded and the field size must grow (as $r \le n$);
in particular, they are vacuous for any fixed finite field.
Nonetheless, the paradigm of \emph{Stillman uniformity} (see e.g.\ \cite{ErmanSaSn19,AnanyanHo20})---which, broadly speaking, posits that a bounded number of polynomials 
of bounded degree exhibits ``bounded complexity''---suggests\footnote{We may indeed assume that $m$ is bounded here: the single polynomial case is in fact equivalent to the general case.} that these bounds can be made independent of $n$, meaning that
the truth is simply
$\delta = \abs{\F}^{-\Omega_k(r)}$. A short linear algebra argument shows that this is indeed the case when $k=2$.

A fundamental goal of higher-order Fourier analysis is to understand bounded functions over finite-field vector spaces that have large Gowers $U^k$-norm; that is, functions $f \colon \F^n \to \CC$ for $\F = \F_p$ a finite field of prime order, for which $\abs{f(x)} \leq 1$ and $\Abs{f}_{U^{k}} \geq \delta$ for some $\delta$. 
Following work of Bergelson, Tao, and Ziegler~\cite{BTZ10,TZ10}, in high characteristics $p \geq k$ it is known that $f$ must correlate with a polynomial phase of degree at most $k-1$; that is, there is a polynomial $Q$ of degree at most $k-1$ such that $\abs{\E_x[f(x)\omega^{-Q(x)}]} \geq \Omega_{k,\abs{\F},\delta}(1)$
where $\omega = e^{2\pi i/p}$.
However, much better bounds are suspected to hold; in particular, the polynomial Gowers inverse conjecture posits that in fact we can always find such a $Q$ with $\abs{\E_x[f(x)\omega^{-Q(x)}]} \geq \delta^{O_{k,p}(1)}$. This is a major open question in the field---with the $k=3$ case recently proved in \cite{Gowers2024}. An important special case (see \cite{BogdanovVi10,GreenTao09}) is when $p > k$ and $\Abs{f(x)}_{U^{k+1}} = 1$; in this case $f$ must be a degree-$k$ polynomial phase, i.e.\ 
$\omega^{P(x)}$ 
for $P$ of degree at most $k$.

Although these two questions---Stillman uniformity in polynomial equidistribution and polynomial inverse theorems for degree-$k$ phases under the $U^k$-norm---come from seemingly dissimilar areas, they are in fact intimately related to each other, with the first implying the second. Moreover, assuming high characteristics ($\operatorname{char} \F > k$), the first question has in recent times found a neater formulation only involving a single multilinear polynomial: the conjecture that the partition and analytic ranks of a tensor are equal up to constant factors.\footnote{As we will soon see, this conjecture can still be stated when the characteristic is low; in this case it is weaker than the polynomial equidistribution question.} In addition to the two already mentioned, comparing partition and analytic rank has further applications in a variety of fields, such as improving bounds for inverse theorems for the $U^k$-norm~\cite{GM20,Tid22,Milicevic22}, universality phenomena for high-rank varieties~\cite{KazhdanZi20B}, list-decoding Reed-Muller codes~\cite{BhowmickLo15}, concentration inequalities for ranks
and decoding by constant-depth circuits~\cite{BrietCa22}, and worst case to average case reduction in complexity theory~\cite{KaufmanLo08}.

Let us now be more precise. For $k \geq 2$, we let a \emph{$k$-tensor} be a multilinear map $T \colon (\F^n)^k \to \F$ over a field $\F$; in other words, a polynomial in $nk$ variables $x_{i,j}$, where $1 \leq i \leq k$ and $1 \leq j \leq n$, such that every monomial is of the form $c x_{1,i_1} x_{2,i_2} \cdots x_{k,i_k}$ with $c \in \F$. The partition rank $\PR(T)$~\cite{Naslund20} is a measure of algebraic structure:
it is the minimum number of $k$-tensors, reducible as polynomials,
needed to sum to $T$.\footnote{So $T = \sum_{i=1}^{\PR(T)} S_iR_i$ with $S_i,R_i$ non-constant. In fact, any reducible multilinear polynomial is the product of two (nonconstant) multilinear polynomials (indeed, any factorization of a multihomogeneous polynomial must have multihomogeneous polynomials as factors). 
This added structure is typically built into the definition of partition rank.}
When $\F$ is finite, the analytic rank $\AR(T)$~\cite{GowersWo11} is a measure of bias or quasirandomness: it is defined as
$\AR(T) = -\log_{\abs{\F}} \bias(T)$,
where $\bias(T) = \Pr_x[T(\vec{x})=0] - \Pr_x[T(\vec{x})=y]$ for any nonzero $y \in \F$.\footnote{Bias is independent of $y$ since the multilinearity of $T$ implies that it is exactly equidistributed on $\F^\times$.}
The ``analytic'' terminology comes from the equivalent definition $\bias(T) = \E_{\vx}[\chi(T(\vx))]$,
where $\chi \colon \F \to \mathbb{C}$ is any nontrivial additive character.\footnote{For example, $\chi(x) = e^{2\pi ix/\abs{\F}}$
when $\abs{\F}$ prime.}
Yet another equivalent definition is especially useful: 
$\bias(T)=\Pr_{\vx\in (\F^n)^{k-1}}[\forall y \in \F^n \colon T(\vx,y) = 0]$.

An important example, already alluded to above, is the case $k = 2$: in this case, the partition rank of a bilinear form $T$ is the minimum number of outer products needed to sum to $T$, which is well-known to be the standard notion of rank from linear algebra, denoted $\rk(T)$. Moreover, $\AR(T) = \rk(T)$ as well, since
\[\bias(T) = \Pr_{x \in \F^n}[T(x,-) = 0] = \abs{\F}^{-\rk(T)}.\]
Further similarities between partition and analytic rank include the facts (none of which are difficult to show) that 
$0 \le \PR(T),\AR(T) \le n$, 
that if $T \neq 0$ is reducible as a polynomial then $\PR(T) = 1$ and $\AR(T) = 1+o(1)$,
and that if $T$ is chosen uniformly at random then $\PR(T) = n$ and $\AR(T) = n - o(1)$ with high probability.\footnote{Here, the $o(1)$ is with respect to the limit $\abs{\F}\to\infty$.}

We now discuss the comparison of partition and analytic rank. As the bound $\AR(T) \le \PR(T)$ is not so difficult~\cite{KazhdanZi18,Lovett19}, it is the inverse question, of bounding $\PR(T)$ from above in terms of $\AR(T)$, that has been the focus of a long line of work. Intuitively, such a bound says that if a tensor is biased then it has to be because of one reason: the tensor is algebraically structured. The influential work of Green and Tao~\cite{GreenTao09}, together with later refinements by Kaufman and Lovett and Bhowmick and Lovett~\cite{KaufmanLo08,BhowmickLo15}, imply that $\PR(T) \le f_{k}(\AR(T))$ for some function $f_k$.
However, these results rely on a regularity lemma for polynomials, and result in $f_k$ that are Ackermann-type.  
More recently, Mili\'cevi\'c~\cite{Milicevic19} and Janzer~\cite{Janzer19} obtained polynomial bounds of the form $\PR(T) = O_{k}(\AR(T)^{c_k})$ where $c_k = \exp(\exp(k^{O(1)}))$.\footnote{The bound in \cite{Janzer19} has an additional dependence on the field size.}
Better bounds have been shown in specific cases, such as when $k=3$ or $k=4$~\cite{HaramatySh10,Lampert19}. Notably, a linear bound $\PR(T) = O_k(\AR(T))$ has been shown in the case $k = 3$~\cite{AdiprasitoKaZi21,CohenMo21,Lampert24},
or over fields of size at least double-exponential in $\AR(T)$~\cite{CohenMo21B}.
As alluded to previously, the ultimate conjecture in this line of work is that this linear bound holds unconditionally:
\begin{conj}[Partition vs.\ analytic rank conjecture \cite{AdiprasitoKaZi21,KazhdanZi20B,LampertZi21,Lovett19}]\label{conj:PR-AR}
For every $k$-tensor $T$ over every finite field, $\PR(T) \leq O_{k}(\AR(T))$.
\end{conj}

In this paper we prove \cref{conj:PR-AR} up to a logarithmic factor. For the remainder of this paper let $\lf_\F(x) = \log_{\abs{\F}}(x+1)+1$.
Our main result is the following.
\begin{main}\label{main:main}
For every $k$-tensor $T$ over every finite field $\F$, 
\[\PR(T) \leq O_{k}(\AR(T) \lf_\F(\AR(T))).\]
\end{main}
\cref{main:main} improves on the polynomial bounds of Mili\'cevi\'c and Janzer \cite{Milicevic19,Janzer19} by not only showing that the exponent can be taken to be independent of $k$, but also that the exponent is $1+o(1)$. 
Moreover, it implies a linear bound under the assumption that $\abs{\F}$ is at least some small power of $\AR(T)$; this improves on \cite{CohenMo21B} where the bound on $\F$ is double exponential.
It improves or subsumes several other results in the literature; see \cref{subsec:apps} below.

\begin{remark}
A previous version of this paper showed $\PR(T) \leq O_{k}(\AR(T) \log^{k-1}(\AR(T)+1))$. The proof there in fact shows $\PR(T) \leq O_{k}(\AR(T) \lf_\F(x)^{k-1})$, though at the time we chose not to mention this dependence on $\F$ as it is irrelevant when $\F$ is fixed. On the other hand, the removal of $k-2$ log factors comes from an improved argument at the very end of the proof, relying on a recent result of Chen and Ye~\cite{ChenYe24,UniformStability} concerning the stability of analytic rank under field extension.
\end{remark}

One notable feature of our proof is that it has little dependence on previous work. Though some ideas are conceptually borrowed from arguments in~\cite{CohenMo21B} using algebraic geometry, the core of the proof (everything in this paper until \cref{sec:mainproof}) relies entirely on linear algebra and extensive work with polynomials, with the most sophisticated tool being the Schwartz-Zippel lemma over finite fields. However, to extract \cref{main:main}, we need to cite a ``non-elementary'' result: namely, the stability of analytic rank under field extensions \cite{ChenYe24,UniformStability}.

\subsection{Applications}\label{subsec:apps}
As mentioned above, \cref{main:main} yields improved bounds for polynomial equidistribution and an almost-polynomial Gowers inverse theorem for polynomial phases, which we now state along with a third corollary regarding the ranks of polynomials.

Given a polynomial $Q$ of degree $d$ over a field $\F$, define its \emph{rank} to be $0$ if $Q$ is constant, $\infty$ if $d = 1$, and if $d \geq 2$, the minimum $r$ such that the degree-$d$ portion of $Q$ can be written as a sum of $r$ reducible homogeneous polynomials. Given a tuple $P = (P_1,\ldots,P_m)$ of polynomials, its \emph{rank} is the minimum rank of any nontrivial linear combination of the $P_i$. This is finite, unless $P$ is a surjective affine map.
\begin{mcor}\label{mcor:polys-equidistribution}
Let $\F$ be a finite field of characteristic greater than $k$ and let $P = (P_1,\ldots,P_m)$ be a tuple of polynomials over $\F$ of degree at most $k$. If $P$ has rank $r < \infty$, then
\[\abs{\Pr_x[P(x) = 0] - \abs{\F}^{-m}} \leq \abs{\F}^{-\Omega_k(r/\lf_\F(r))}.\]
\end{mcor}
This bound (see \cref{remark:Schmidt-subsume}) subsumes the aforementioned earlier bound of Schmidt \cite{Schmidt84,Schmidt85}.

Another application involving polynomial rank is a nearly linear relationship between rank over a finite field and rank over the algebraic closure. Specifically, let $\brk(Q)$ denote the rank of $Q$ when treated as a polynomial over $\overline \F$.
Clearly, $\brk(Q) \le \rk(Q)$, and it is an important conjecture that---analogously to \cref{conj:PR-AR}---the reverse inequality holds up to a constant: $\brk(Q) \le O_k(\brk(Q)))$, provided $Q$ is multilinear or $\ch(\F)>k$ (see e.g.~\cite[Conjecture~1.7]{AdiprasitoKaZi21} and~\cite{LampertZi24}).
The best known bound~\cite{LampertZi24} over finite fields is of the general form $\rk(Q) \le O(\brk(Q)^{c_k})$ with $c_k=\exp(\exp(k^{O(1)}))$ (mirroring the bounds obtained by~\cite{Milicevic19,Janzer19} for the partition-vs-analytic rank problem). \cref{main:main} shows the following.
\begin{mcor} \label{mcor:barpr}
Let $\F$ be a finite field and let $Q$ be a polynomial over $\F$ of degree $k$. Then, if $Q$ is multilinear or $\ch(\F) > k$, we have
\[\rk(Q) \leq O_k(\brk(Q) \lf_\F(\brk(Q))).\]
\end{mcor}

For a finite field $\F = \F_p$ of prime order and a function $f \colon \F^n \to \CC$, define the Gowers $U^k$-norm $\norm{f}_{U^k}$ and weak Gowers $u^k$-norm $\norm{f}_{u^k}$ by
\[\norm{f}_{U^k} = \abs{\E_{v_1,\ldots,v_k,x\in\F^n} \Delta^*_{v_1}\cdots\Delta^*_{v_k} f(x)}^{1/2^k}
\quad\text{ and }\quad
\norm{f}_{u^k}=\max_{\deg(Q)<k} \abs{\E_{x\in\F^n} f(x)\omega^{-Q(x)}},\]
where we let $\Delta^*_{v}f(x)=f(x+v)\overline{f(x)}$ and $\omega = e^{2\pi i/p}$.
Next, abbreviate $\lf_p(x) = \lf_{\F_p}(x)$.
\begin{mcor} \label{mcor:pgi}
For prime $p > k$ and every degree-$k$ polynomial phase $f \colon \F_p^n \to \mathbb{C}$, if $\norm{f}_{U^k} \ge \delta$ $(>0)$ then
\[\norm{f}_{u^k} \ge \d^{O_k(\lf_p(\log_p(1/\delta)))}
    = \exp(-\log(1/\d)^{1+o_{k}(1)}).\]
\end{mcor}

\subsection*{Local rank}
The main theoretical contribution of this paper is a notion that we call \emph{local rank}, denoted $\LR_\vp(T)$, which, as its name suggests, measures the complexity of a tensor $T$ ``near'' a point $\vp$. 
Unlike other notions of rank for tensors, the local rank is a tuple (of nonnegative integers).
To define it, it is convenient to shift notation by letting $d = k-1$ and $V = \F^n$. We now work in the language of $d$-linear maps $T \colon V^d \to V$, which naturally correspond to $(d+1)$-tensors (via $(x_1,\ldots,x_{d+1})\mapsto T(x_1,\ldots,x_d) \cdot x_{d+1}$). Given such a $T$ and some $\vx \in V^d$, we let $M_i(\vx)\coloneq T(x_1,\ldots,x_{i-1},-,x_{i+1},\ldots,x_d)$ be the $V\to V$ linear map obtained by fixing all inputs except the $i$th to their value in $\vx$.

Now, given some $\vp = (p_1, p_2, \ldots, p_d) \in V^d$, consider the following process, which can be visualized as a walk in $V^d$:
\begin{itemize}
    \item Pick $x_d$ such that $T(p_1, p_2, \ldots, p_{d-1}, x_d) = 0$, i.e.\ $x_d \in \ker M_d(\vy_d)$ where $\vy_d = \vp$.
    \item Pick $x_{d-1}$ such that $T(p_1, p_2, \ldots, p_{d-2}, x_{d-1}, x_d) = 0$, i.e.\ $x_{d-1} \in \ker M_{d-1}(\vy_{d-1})$ where $\vy_{d-1} = (p_1, p_2,\ldots,p_{d-1},x_d)$.
    \item etc., until
    \item Pick $x_1$ such that $T(x_1, x_2,\ldots,x_d) = 0$, i.e.\ $x_1 \in \ker M_1(\vy_1)$ where $\vy_1 = (p_1,x_2,\ldots,x_d)$.
\end{itemize}
Now consider the tuple of ranks
\[\vr = (\rank M_1(\vy_1), \rank M_2(\vy_2), \ldots, \rank M_d(\vy_d)).\]
The \emph{local rank} of $T$ at $\vp$, denoted $\LR_\vp(T)$, is defined to be the colexicographic\footnote{Colexicographic order means that tuples are ordered first by their last component, and then tiebroken by their next-to-last component, and so on.} maximum of $\vr$ over all possible ways to run this process (see \cref{def:LR}; see also \cref{prop:lrdefrecur} for an equivalent, recursive definition).
As we discuss later, the choice of colexicographic ordering is essential (see \cref{ex:lr3}), since over an infinite field, it makes the set of all tensors with (colexicographically) bounded local rank Zariski closed. 
In fact, a crucial step in our proof is the construction of polynomials of bounded degree, generalizing matrix minors, that cut out this variety (see \cref{prop:ddeg}).
This should be compared with the situation for partition rank---it is believed (see \cite{BallicoBiOnVe22, DraismaKar24}) that the set of $k$-tensors of bounded partition rank is not Zariski closed for $k \ge 4$. 

As we will be working with finite fields, one technicality we will have to deal with is the distinction between the local rank $\LR_\vp(T)$ defined above and the \emph{algebraic local rank}, denoted $\bLR_\vp(T)$, which is simply the local rank of $T$ when $T$ is viewed as a tensor over the algebraic closure $\overline\F$ of $\F$ (in fact, any infinite extension of $\F$ would do). Intuitively, the difference between $\LR_\vp(T)$ and $\bLR_\vp(T)$ is that for the latter, the aforementioned 
points $\vy_i$ are
no longer constrained to be points over $\F$.

In a sense, algebraic local rank can be thought of as a far-reaching generalization of the notion of \emph{commutative rank}. Over a large enough finite field, this is simply the maximum rank of a matrix of linear polynomials (and which has strong ties to topics in areas ranging from algebraic complexity to invariant theory and quantum information theory~\cite{FortinRe04,Wigderson17}).
In fact, for $\vp=0$ and $d=2$, local rank specializes to maxrank, and algebraic local rank specializes to commutative rank, 
as $\LR_{\vec{0}}(T) = (\max_{\vx \in V^d} \rank M_1(\vx),0,\ldots,0)$ is given by the maximum rank of a matrix of multilinear polynomials.

The definitions of local and algebraic local rank lead to the concept of \emph{LR-stable} points: we call $\vp$ \emph{LR-stable} for $T$ if $\LR_\vp(T)=\bLR_\vp(T)$.
LR-stable points can be viewed as analogues of nonsingular points in the sense of algebraic geometry, as both are defined using ``local niceness'' and enable certain analyses in their neighborhood, except its definition is combinatorial and it depends only on the polynomials cutting out the variety (the components of $T$) rather than on the entire radical ideal they generate.

\begin{remark}
    Local rank has applications beyond the scope of the results in this paper,
    and we believe it could be a powerful tool in studying difficult questions about polynomials.
    Let us briefly mention one algebraic-geometric application, communicated to us by A.~Lampert.
    It says, roughly, that any multilinear variety contains a smooth $\F$-point (i.e., with coordinates in $\F$ rather than $\overline{\F}$) on some subvariety whose codimension is bounded in terms of local rank at an LR-stable point.
    \begin{prop}[\cite{LampertApp}]
             Suppose $\vp$ is an LR-stable point for a $d$-linear map $T=(T_1,\ldots,T_n) \colon V^d\to V$ over $\F$. Then there is a variety $X \subseteq V(T_1,\ldots,T_n)$ with $\codim_{V^d}(X) \le \norm{\LR_\vp(T)}_1$ and such that $X$ contains a smooth $\F$-point.
    \end{prop}
\end{remark}
Let us also remark that, beyond the applicability of local rank, the techniques used in this paper (say in \cref{sec:LR-AR}) hint at a general method for reducing algebro-geometric statements over algebraically closed fields of positive characteristic to statements about quite small fields (by finding a ``witness'', an LR-stable point, in a small field).

We refer the reader to \cref{sec:lr} for further discussion of local rank, which includes a series of examples showing that, in a sense, the definition of local rank is forced---changing it in any of several ways results in losing one or more of its properties.

\subsection*{Proof overview}
Having defined local rank, the bulk of our proof now splits into two parts: bounding local rank above by the analytic rank, and bounding the partition rank above by the local rank. 

Recall that we call $\vp$ an LR-stable point if $\LR_\vp(T)=\bLR_\vp(T)$.
Our first main technical result finds an LR-stable point where the norm of the local rank vector is bounded from above in terms of the analytic rank. Crucially, this LR-stable point is found already over fields that are relatively small: of polynomial size in the analytic rank.
In our results we use the norm $\norm{(v_1,\ldots,v_d)} \coloneqq \sum_{i=1}^d 2^{d-i}v_i$.
\begin{theo}[$\LR$ vs.\ $\AR$]\label{main:LR-AR}
For every $k \geq 2$ there is some $C_k>0$ such that, for every  
$\eps \in (0,1]$
and $k$-tensor $T$ over a finite field $\F$, there is an LR-stable point $\vp$  
with $\norm{\LR_\vp(T)} \le (2^{k-1}-1+\eps)\AR(T)$, provided $\abs{\F}\geq C_k(\AR(T)+1)^{k-2}/\eps$.
\end{theo}

The proof of \cref{main:LR-AR} analyzes the local rank by considering a certain type of random walk on the zero set of $T$.
A concavity-based argument shows that, in expectation over the zero set of $T$, the rank of a matrix slice is bounded above by $\AR(T)$,  
and by applying the random walk we are able to control an approximate version of local rank at 
an appropriate
point $\vp$. This approximate local rank is shown to be equal to $\LR_\vp(T)$ and to $\bLR_\vp(T)$, provided the field is large enough.
This relies on the fact that the set of $k$-tensors of bounded algebraic local rank can be defined by the vanishing of polynomials of low degree.

Our second main technical result shows that, at an LR-stable point, the matrix slices used in the definition of local rank are enough to reconstruct the tensor (and ``efficiently'' so).
\begin{theo}[$\LR$ vs.\ $\PR$] \label{main:LR-PR}
If $\vp$ is an LR-stable point for a tensor $T$ then $\PR(T) \leq \Abs{\LR_\vp(T)}$.
\end{theo}

An important idea in the proof of \cref{main:LR-PR} is treating partition rank decompositions as
algebraic objects, on which we can apply algebraic operations (to construct decompositions for tensors of higher order). 
Crucially,
we prove by induction a stronger statement than the existence of partition rank decompositions---roughly, we construct a rational formula which gives such a decomposition for every $k$-tensor of bounded local rank at $\vp$.
The basis of the induction, the case where $k=2$, uses classical identities from linear algebra.
In the induction step, we consider ``slices'' of a $k$-tensor with bounded local rank obtained by fixing the last input. As these slices have small local rank, by the inductive hypothesis they can be expressed by a formula. To convert this into a formula for $k$-tensors, we take a derivative, utilizing the fact that the derivative of a multilinear polynomial is, in some sense, itself. The importance of LR-stability here comes from the need to eliminate the possibility that two polynomials agree at all points over $\F$ but have different derivatives, which is impossible if the two polynomials agree over $\overline\F$.

Combining \cref{main:LR-AR} and \cref{main:LR-PR} immediately proves the following result.
\begin{theo} \label{theo:main1}
For every $k \geq 2$ there is $C_k>0$ such that, for every $\eps \in (0,1]$ and $k$-tensor $T$ over a finite field $\F$, we have $\PR(T) \leq (2^{k-1} - 1 + \eps) \AR(T)$, provided $\abs{\F} \geq C_k(\AR(T)+1)^{k-2}/\eps$.
\end{theo}
The statement \cref{main:main}, which applies to all finite fields, follows from a quick argument that passes to a sufficiently large field extension.

Finally, we remark that although the constant $2^{k-1} - 1$ in \cref{theo:main1} does not meaningfully affect the statement of \cref{main:main}, it is curious that it is the same constant that appears in the main result of~\cite{CohenMo21B}.

\subsection*{Outline} 
In \cref{sec:prelim} we give preliminaries relating to tensors and define formal rational maps, a useful tool for keeping track of various formulas.
In \cref{sec:Linear-Algebra} we generalize certain identities from linear algebra which will be used later to both prove important properties of (algebraic) local rank (in \cref{sec:lr}), and to construct partition rank decompositions of higher-order tensors (in \cref{sec:prlr}).
In \cref{sec:3tensor} we outline our proof in the special case where $k = 3$, which motivates the definition of local rank and our subsequent arguments; the reader interested only in this special case may skip \cref{sec:Linear-Algebra} except for \cref{cor:algrank}.
The core of the proof is found in \cref{sec:lr,sec:LR-AR,sec:prlr}:
\cref{sec:lr} formally defines local rank and develops some of its properties;
\cref{sec:LR-AR,sec:prlr} prove \cref{main:LR-AR,main:LR-PR}, respectively.
We conclude with \cref{sec:mainproof}, which proves \cref{main:main} and its corollaries.

\section{Preliminaries} \label{sec:prelim}
\subsection{Vector and tensor spaces}
In \cref{sec:prelim,sec:Linear-Algebra,sec:3tensor,sec:lr,sec:LR-AR,sec:prlr} of this paper, we will work with tensors using the following language: let $\F$ be a field and let $V$ be a finite-dimensional vector space over $\F$. 
Then, for a nonnegative integer $d$, a \emph{tensor} is an element of the set
\[\cT^d = \{T\colon V^d \to V \mid \text{$T$ is a $d$-linear map} \} .\]
This notion of tensor corresponds to the definition mentioned in the introduction in the following way: 
if $T \colon (\F^n)^k \to \F$ is a multilinear polynomial, 
then, setting $V = \F^n$, it corresponds to 
the $(k-1)$-linear map
in $\cT^{k-1}$ given by
$\vx \mapsto (T(\vx, e_1), \ldots, T(\vx, e_n))$,
where $\vx \in (\F^n)^{k-1}$ 
and $e_1,\ldots,e_n \in \F^n$ is the standard basis; 
importantly, elements of $\cT^d$ are $(d+1)$-tensors.
For $T \in \cT^d$, we denote the zero set of $T$ (over $\F$) by
\[\Z(T)=\{ \vx \in V^d \mid T(\vx)=0 \} .\]

Let $\overline \F$ denote the algebraic closure of $\F$.\footnote{One can in fact adapt the proofs in this paper to work with any infinite extension of the ground finite field $\F$ instead of the algebraic closure of $\F$.}
For an $\F$-vector space $W$, let $\overline W$ be the
corresponding $\overline \F$-vector space, 
$W \otimes_\F \overline \F$, treated as a superset of $W$. 
Note that $\overline \cT^d$ can be identified with the set of all $d$-linear maps $T \colon \overline V^d \to \overline V$. 
Let $\Id \in \cT^1$ denote the identity linear map. 
Note that if $T \in \overline \cT^1$ then $T$ is just a linear map, so it makes sense to talk about the rank of $T$; we denote this as $\rank T$. (There is no ambiguity here since rank is invariant under field extension.) Additionally, we let $\ker T$ denote the $\F$-kernel of $T \in \cT^1$, while we let $\bker T$ denote the $\overline \F$-kernel of $T \in \overline \cT^1$. Define $\im T$ and $\bim T$ similarly as the images of $T$ over $\F$ and $\overline \F$.

\subsection{Notation for working with higher-order tensors} \label{subsec:tensornotation}
\renewcommand{\ominus}{\circ}

Let $\NN$ be the set of nonnegative integers,
and for $n \in \NN$ let $[n] = \{1,\ldots,n\}$. 
In this paper, bold upright letters will 
represent a tuple of about $d$ objects, which may be vectors, numbers, or something else. These can always be indexed; that is, $x_i$ will always represent the $i$th element of $\vx$, provided it exists. 
Importantly, we let $\vx'$ denote the tuple obtained from $\vx$ by removing the last coordinate, that is,
\[(x_1,\ldots,x_d)' \coloneqq (x_1,\ldots,x_{d-1}) .\]
We use the following notation for operations on tuples:
\begin{itemize}
    \item For $\vx, \vy \in V^d$ and $0 \leq i \leq d$, let
    \[\vx \ominus_i \vy \coloneqq (x_1,\ldots,x_i,y_{i+1},\ldots,y_d).\]
    That is, $\vx \ominus_i \vy$ is the concatenation $(x_1,\ldots,x_i) \circ (y_{i+1},\ldots,y_d)$.
    \item For $\vx \in V^d$, $v \in V$, and $i \in [d]$, let
    \[\vx \odot_i v \coloneqq (x_1,\ldots,x_{i-1},v,x_{i+1},\ldots,x_d).\]
    For example, $\vx \odot_i x_i = \vx$.
\end{itemize}
We use the following notation for partial assignments to tensors $T \in \cT^d$:
\begin{itemize}
    \item For $v \in V$, let
    \[T[v] \coloneqq T(-,\ldots,-,v) \in \cT^{d-1}.\]
    That is, $T[v] \colon V^{d-1} \to V$ is the tensor slice given by the $(d-1)$-linear map $T[v] \colon \vx \mapsto T(\vx,v)$.
    Note that $T[-]\colon V\to \cT^{d-1}$ is a linear map (i.e., $T[v]$ is linear in $v$).
    \item For $\vx \in V^d$ and $i \in [d]$, let
    \[T[\vx]_i \coloneqq T(x_1,\ldots,x_{i-1},-,x_{i+1},\ldots,x_d) \in \cT^1.\]
    That is, $T[\vx]_i \colon V \to V$ is the ``matrix slice'' given by the linear map $T[\vx]_i \colon v \mapsto T(\vx \odot_i v)$.\footnote{Note that $T[\vx]_i$ does not depend on $x_i$;
    we make this notational choice in order to preserve the indexing.}
    Note that $T[-]_i \colon V^d \to \cT^1$ is a polynomial map (in fact, a $(d-1)$-linear map).
\end{itemize}

\subsection{Polynomial and formal rational maps}
Let $U$ and $W$ be $\F$-vector spaces. 
We let $\Poly(V)$ denote the set of polynomials $f \colon V \to \F$.
More generally, we let $\Poly(U, W)$ denote the set of polynomial maps from $U$ to $W$ (so that $\Poly(V)=\Poly(V,\F)$). 
If $f \in \Poly(U,W)$, $\deg f$ denotes the 
total degree\footnote{Meaning, for $f=(f_1,\ldots,f_m)$, the largest total degree of the polynomials $f_i$. In practice all the polynomials we deal with will be homogeneous, but this is not a crucial detail to keep track of.} of $f$.
Any $f \in \Poly(U,W)$ induces a function $\overline U \to \overline W$, which we identify with $f$.

In this paper, we construct formulas that take a tensor and output a partition rank decomposition of that tensor. These formulas are constructed by applying algebraic operations on other formulas, 
which can be thought of as rational maps from tensors to tensors. 
However, for technical reasons,\footnote{This stems from 
the fact that the domain of a rational function is a bit of a murky notion. 
Defining the domain maximally, by first reducing to lowest terms,
can have an unexpected implication when applying algebraic operations.
Specifically, given rational maps $R_1$ and $R_2$ with respective domains $U_1$ and $U_2$, the domain of their product $R_1R_2$ (or any algebraic expression involving $R_1$ and $R_2$) may be larger than $U_1 \cap U_2$. Thus, even if we show that $R_1$ and $R_2$ satisfy some useful property on their domains, it can be difficult to say the same for $R_1R_2$ at points outside $U_1 \cap U_2$.
Fundamentally, this is because a rational function is defined using an equivalence relation that allows one to freely cancel common factors between the numerator and denominator.
We opt to do away with this equivalence relation, which makes the domain always behave in the expected way.} 
we keep track of both the numerator and denominator separately, without them interacting, that is, without canceling common factors.
Treating a numerator-denominator pair as a single algebraic object, it turns out that we can define enough
operations, such as addition, multiplication, and composition, so that the reader uninterested in technical details may ignore the distinction between rational and formal rational maps when reading the proofs in this paper.

Formally, for $\F$-vector spaces $U$ and $W$, we define $\FRat(U,W) \coloneq \Poly(U,W) \times \Poly(U)$, the set of \emph{formal rational maps} from $U$ to $W$.
Elements of $\FRat(U,W)$ will be denoted using uppercase sans-serif letters and expressed using a thick fraction bar, as in $\rat{R}=\ffrac{F}{g}$, where $F \in \Poly(U,W)$ is a polynomial map and $g \in \Poly(U)$ is a (single) polynomial. 
In conceptualizing a formal rational map $\ffrac{F}{g}$, it may be helpful to think of it as an ordered pair consisting of a standard rational map\footnote{To be very precise, we do not forbid the denominator from being the zero polynomial (there is no harm in doing so as such a formal rational map is everywhere undefined and therefore never used for anything).} $F/g$ together with a subset of its domain on which we care about the behavior of the map (that subset being $\setmid{u \in U}{g(u) \neq 0}$, noting that $g$ is otherwise forgotten once we switch to the standard rational map).
However, in a sense, a formal rational map is actually a simpler object than a standard rational map, in that in working with it we are able to ignore completely the intricacies caused by polynomial factorization.\footnote{It is possible that the numerator and denominator in the formulas of this paper never share common factors outside trivial cases, making formal rational maps unnecessary. However, it is unclear how to prove this.}

For $u \in \overline U$, $w \in \overline W$,
and $\rat{R} = \ffrac{F}{g} \in \FRat(U,W)$, we say that $\rat{R}$ is \emph{defined} at $u$ if $g(u) \neq 0$. 
We write $\rat{R}(u) \equiv w$ if $F(u) = g(u)w$;
note that if $F(u)$ and $g(u)$ are both zero, then $\rat{R}(u)\equiv w$ for all $w$.
More generally, for $H \in \Poly(U,W)$, we write $\rat{R} \equiv H$ if $F = gH$.

One can define a number of algebraic operations on formal rational maps in analogy with rational maps. For a given formal rational map $\rat{R} = \ffrac{F}{g} \in \FRat(U,W)$, define the following:
\begin{itemize}
\item If $\rat{R'} = \ffrac{f'}{g'} \in \FRat(U,\F)$ then $\rat{R'}\rat{R} \coloneqq \ffrac{f'F}{g'g}$.
\item If $H \in \Poly(U,W)$ then $\rat{R}+H = H +\rat{R} \coloneqq \ffrac{F+gH}{g}$.
\item If $H \in \Poly(U',U)$, where $U'$ is another $\F$-vector space, then $\rat{R} \circ H \coloneqq \ffrac{F \circ H}{g \circ H}$.
\end{itemize}
Later in this paper (\cref{subsec:derivatives}) we will further define the derivative of a formal rational map.

Note that $\F$ embeds in $\Poly(U)$ in the obvious way, which subsequently embeds in $\FRat(U,\F)$ by sending $h \mapsto \ffrac{h}{1}$, so we have also defined the multiplication of an element in $\FRat(U,W)$ by a constant or an element in $\Poly(U)$. 

\subsection{Analytic and partition rank of tensors}
One can translate the definitions of analytic and partition rank from multilinear polynomials
to multilinear maps $T \in \cT^d$.
If $\F$ is finite, the \emph{analytic rank} of $T \in \cT^d$ is $\AR(T) \coloneqq -\log_{\abs{\F}} (\abs{\Z(T)}/\abs{V}^d)$.

Let $\cM^k$ denote the set of $k$-linear polynomials $V^k \to \F$ (i.e.\ degree-$k$ multilinear polynomials).  
A nonzero tensor $T \in \cT^d$ is \emph{reducible}
if its components have a (nontrivial) common factor,
or equivalently, if there is a partition $[d] = A \sqcup B$ with $A$ nonempty,
along with a polynomial $P \in \cM^{\abs{A}}$ and a map $T' \in \cT^{\abs{B}}$ (so $T' \colon V^{\abs{B}}\to V$), 
such that $T(\vx) = P(x_i \colon i \in A) \cdot T'(x_j \colon j \in B)$ 
for every $\vx \in V^d$. 
In the future, we will at times abuse notation and abbreviate the above relation as $T = P \otimes T'$, hiding the dependence on $A,B$.
The \emph{partition rank} $\PR(T)$ of $T$ is the minimum $r$ such that $T$ can be written as the sum of $r$ reducible tensors.

In this paper, it will be important to additionally define the partition rank of tensors that are specified by a polynomial map. A nonzero polynomial map $F \in \Poly(W,\cT^d)$, where $W$ is an arbitrary $\F$-vector space,  
is \emph{reducible} 
if its components have a (nontrivial) common factor,
or equivalently,
if there is a partition $[d] = A \sqcup B$ with $A$ nonempty, along with polynomial maps $F_1 \in \Poly(W, \cM^{\abs{A}})$ and $F_2 \in \Poly(W ,\cT^{\abs{B}})$, such that
\[F(w)(\vx) = F_1(w)(x_i \colon i \in S) \cdot F_2(w)(x_j \colon j \in B)\]
as polynomials in $w$ and $\vx$.
As above, $\PR(F)$ is the minimum number of reducible polynomial maps needed to sum to $F$. 

An observation that we will use repeatedly is that if $G \in \Poly(W',W)$ is a polynomial map, then 
$\PR(F \circ G) \leq \PR(F)$. 
Finally, for a formal rational map into $\cT^d$, that is, $\ffrac{F}{g}\in \FRat(W, \cT^d)$,
we define 
$\PR(\ffrac{F}{g}) = \PR(F)$.

\section{Polynomial Identities from Linear Algebra}\label{sec:Linear-Algebra}

In this section we will generalize certain classical matrix identities.
Let $\F^{m \times n}$ be the vector space of $m \times n$ matrices with entries in $\F$,
and let $\Id_n \in \F^{n \times n}$ denote the $n\times n$ identity matrix. 
Given a matrix $A \in \F^{m \times n}$ and subsets $I \subseteq [m]$ and $J \subseteq [n]$, let $A_{I \times J}$ denote the $\abs{I} \times \abs{J}$-submatrix of $A$ consisting of the rows with indices in $I$ and columns with indices in $J$. 
Given a square matrix $A \in \F^{n \times n}$, let $\adj(A) \in \F^{n \times n}$ denote the adjugate matrix of $A$,
and recall the identity $A\adj(A) = \adj(A) A = \det(A)\Id_n$.
We will frequently use the fact that a matrix has rank smaller than $r$ if and only if all its $r \times r$ minors are zero.

\begin{lemma}\label{lemma:adj}
Let $A\in \F^{m\times n}$, $I\in \binom{[m]}{r}$, and $J \in \binom{[n]}{r}$. 
If $\rank A \le r$ then 
\begin{equation}\label{eq:adj-weak}
    A_{[m] \times J} \adj (A_{I \times J}) A_{I \times [n]} = \det(A_{I \times J}) A.
\end{equation}
If $\rank A < r$ then, moreover, 
\begin{equation}\label{eq:adj-strong}
    \adj (A_{I \times J}) A_{I \times [n]}=0.
\end{equation}
\end{lemma}
\begin{proof}
First, we claim that if $\rank A_{I \times J}=\rank A_{I \times [n]} < r$ then~\labelcref{eq:adj-strong} holds, which in turn implies~\labelcref{eq:adj-weak} since $\det A_{I \times J} = 0$.
Indeed, if $\rank A_{I \times J}=\rank A_{I \times [n]}$ then every column of $A_{I \times [n]}$ is a linear combination of the columns of $A_{I \times J}$, meaning there exists a matrix $M \in \F^{r \times n}$ such that 
$A_{I \times [n]} =  A_{I \times J}M$;
thus, $\adj(A_{I \times J})A_{I \times [n]} =  \det(A_{I \times J})M = 0$, as claimed.

Second, we claim that if $\rank A=\rank A_{I \times [n]}$ then~\labelcref{eq:adj-weak} holds.
Indeed, if $\rank A=\rank A_{I \times [n]}$ then every row of $A$ is a linear combination of the rows of $A_{I \times [n]}$, meaning there exists a matrix $M \in \F^{m \times r}$ such that $A = MA_{I \times [n]}$; 
restricting to the columns in $J$ implies $A_{[m] \times J} = MA_{I \times J}$, and so
$A_{[m] \times J}\adj(A_{I \times J}) = \det(A_{I \times J})M$; multiplying by $A_{I \times [n]}$ on the right
immediately gives
$A_{[m] \times J}\adj(A_{I \times J})A_{I \times [n]} = \det(A_{I \times J})A$, as claimed.

We will repeatedly use the inequalities
$\rank A_{I \times J} \le \rank A_{I \times [n]} \le \rank A$.
We will also henceforth assume $\rank A_{I \times J} \ge r-1$, since otherwise $\adj A_{I\times J} = 0$, as its entries are $(r-1)$-minors, so the stronger guarantee~\labelcref{eq:adj-strong} holds regardless of $\rk A$.

Suppose $\rk A < r$. 
Then $\rk A_{I \times J} = \rk A_{I \times [n]} = r-1$ $(= \rk A)$, 
so~\labelcref{eq:adj-strong} holds by the first claim above.
Suppose $\rk A = r$.
If $\rank A_{I \times [n]} < \rank A$ then $\rk A_{I \times J}=\rk A_{I \times [n]}=r-1$, 
so~\labelcref{eq:adj-strong} again holds by the first claim above.
Otherwise, $\rank A_{I \times [n]} = \rank A$, 
so~\labelcref{eq:adj-weak} holds by the second claim above.
This completes the proof.
\end{proof}

We use~\cref{lemma:adj} to construct (a family of) rational formulas that output a rank-$r$ decomposition (or rank factorization) given a linear map of rank $r$, and which are guaranteed to effectively output $0$ when given a linear map of smaller rank.
In what follows, we choose a basis for $V$, so that $\cT^1=\cT^1(V,\F)$ is identified with $\F^{n\times n}$ where $n=\dim V$.
We also use the fact that adjugation is a polynomial map on matrices, 
$\adj \in \Poly(\F^{r\times r}, \F^{r\times r})$.

\begin{prop}\label{cor:algrank}
    For every $r \in \NN$ there is a set $\cF_{r} \subseteq \FRat(\cT^1, \cT^1)$ satisfying the following properties: 
    \begin{enumerate}
        \item $\PR(\rat{R}) \leq r$ for all $\rat{R} \in \cF_r$;\\
        For all $A \in \overline\cT^1$:
        \item\label{item:F-mat-equiv} If $\rk A \le r$, then $\rat{R}(A) \equiv A$ for all $\rat{R} \in \cF_r$;
        \item\label{item:F-mat-def} If $\rk A = r$, then there is some $\rat{R} \in \cF_{r}$ that is defined at $A$.
    \end{enumerate}
\end{prop}
\begin{proof}
    For every $I,J \in \binom{[n]}{r}$, define 
    $\rat{R}_{I,J} \in \FRat(\cT^1, \cT^1)$ as
    \begin{equation}\label{eq:matrix-PR-formula}
        \rat{R}_{I,J}(A) = \ffrac{A_{[n] \times J} \adj (A_{I \times J}) A_{I \times [n]}}{\det A_{I \times J}}.   
    \end{equation}
    Let $\cF_{r}=\{ \rat{R}_{I,J} \mid I,J \in \binom{[n]}{r}\}$.
    Note that $\PR(\rat{R}_{I,J}) \le r$ since
    \[ A_{[n] \times J} \adj (A_{I \times J}) A_{I \times [n]}
    = \sum_{i=1}^r (A_{[n] \times J})_{[n] \times \{i\}} (\adj (A_{I \times J}) A)_{\{i\} \times [n]}.\]
    If $\rk A \le r$ then, by \cref{lemma:adj}, $\rat{R}_{I,J}(A) \equiv A$.
    If $\rk A = r$ then $\rat{R}_{I,J}(A)$ is defined for any $I,J \in \binom{[n]}{r}$ satisfying $\rk(A_{I,J})=r$, since $\det(A_{I \times J}) \neq 0$.
    This completes the proof.
\end{proof}

Another use of \cref{lemma:adj} is in the construction of a family of projection maps, given here in \cref{def:Proj} and \cref{cor:algproj}, which will be useful in \cref{subsec:reduction-op,subsec:rigid}. As mentioned in the outline, the reader interested in \cref{sec:3tensor} can skip the rest of this section. 

Given $A \in \F^{m\times n}$, $I\in \binom{[m]}{r}$ and $J \in \binom{[n]}{r}$ for $r \le \min\{m,n\}$,
we define a pair of square matrices $P_{I,J}(A), Q_{I,J}(A) \in \F^{n \times n}$ by
\begin{equation}\label{eq:P-Q}
    Q_{I,J}(A) \coloneqq \begin{pmatrix} \adj (A_{I \times J}) A_{I \times [n]} \\ 0\end{pmatrix}
    \quad\text{ and }\quad
    P_{I,J}(A) \coloneqq \det(A_{I \times J})\, \Id_n - Q_{I,J}(A),
\end{equation}
where $\begin{psmallmatrix} M \\ 0 \end{psmallmatrix}$ denotes the 
$n \times n$
matrix obtained by inserting $M \in \F^{r \times n}$ into the rows indexed by $J$ (not $I$!).

\begin{lemma}\label{lemma:P-Q}
Let $A \in \F^{m\times n}$, $I\in \binom{[m]}{r}$, and $J \in \binom{[n]}{r}$. 
If $\rank A \le r$ then $\im P_{I,J}(A) \sub \ker A$, with equality if $\det A_{I \times J} \neq 0$.
If $\rank A < r$ then, moreover, $P_{I,J}(A) = Q_{I,J}(A) = 0$. 
\end{lemma}
\begin{proof}
Suppose $\rank A < r$. 
By \cref{lemma:adj} we have $Q_{I,J}(A)=0$. Moreover, since $\det A_{I \times J} = 0$, we have $P_{I,J}(A) = 0$.

Suppose $\rank A = r$. 
By construction, $AQ_{I,J}(A) = A_{[m] \times J}(Q_{I,J}(A))_{J \times [n]} = A_{[m] \times J} \adj(A_{I \times J}) A_{I \times [n]}$.
Thus, by \cref{lemma:adj}, $AQ_{I,J}(A) = \det(A_{I \times J})A$, which implies that $AP_{I,J}(A) = 0$, that is, $\im P_{I,J}(A) \sub \ker A$. 
Suppose further that $\det(A_{I\times J}) \neq 0$.
By construction, we have $\rank P_{I,J}(A) \ge n - r$.
Since $\im P_{I,J}(A) \sub \ker A$, we have 
$\rank P_{I,J}(A) = \dim(\im P_{I,J}(A)) \le \dim(\ker A) = n-\rk A = n-r$, which implies $\im P_{I,J}(A) = \ker A$. This completes the proof.
\end{proof}

\begin{example}\label{ex:diagonal-proj}
    Let $D$ be diagonal, say $D=\diag(d_1,...,d_k,0,...,0)$ with $d_i \neq 0$,
    and put $c=d_1\cdots d_k$.
    For $I=J=[k]$, the submatrix $D'=D_{I \times J}$ satisfies
    $\adj(D') = c (D')^{-1}$,
    so
    \[Q_{I,J}(D) = c
    \begin{psmallmatrix}
            I_k & 0\\
            0 & 0
    \end{psmallmatrix}
    \quad\text{ and }\quad
    P_{I,J}(D) = c
    \begin{psmallmatrix}
            0 & 0\\
            0 & I_{n-k}
    \end{psmallmatrix}.\]
    We then have $\ker D = \Span\{e_i \mid i\ge n-k\} = \im P_{I,J}(D)$.\\
    For $I=J=[k+1]$ however, $D'=D_{I \times J}$ has 
    $\adj(D') = \diag(0,\ldots,0,c)$, so $\adj(D') D_{I \times [n]} = 0$,
    implying that $Q_{I,J}(D) = 0$ and $P_{I,J}(D) = 0$.
    More generally, $|I|=|J|>k$ implies $P_{I,J}(D) = 0$.
\end{example}

We next state the above result in terms of families of formal rational maps
into $\cT^1$,
which will be useful later.
We choose a basis for $V$, so that $\cT^1$ is identified with $\F^{n\times n}$ where $n=\dim V$.
Crucially, since adjugation is a polynomial map on matrices, 
so are $P_{I,J},Q_{I,J} \in \Poly(\F^{m\times n}, \F^{n \times n})$.

\begin{definition}\label{def:Proj}
    For every $r \in \NN$, let $\cP_{r} \sub \FRat(\cT^1, \cT^1)$ be the set of formal rational maps
    \[\cP_{r}
    = \bigg\{ A \mapsto \ffrac{P_{I,J}(A)}{\det (A_{I \times J})} \,\bigg\vert\, I,J \in \binom{[n]}{r} \bigg\}
    .\]
\end{definition}

In the following result, for a formal rational map $\rat{P} \in \cP_{r}$, we use the italicized $P$ to denote its numerator (which is a polynomial map $P \in \Poly(\cT^1, \cT^1)$).

\begin{coro}\label{cor:algproj}
    For every $r \in \NN$, the set $\cP_{r}$ satisfies the following properties:
    \begin{enumerate}
    \item For all $\rat{P} \in \cP_r$, $\PR(\Id_n-\rat{P}) \le r$ and $P$ is homogeneous of degree $r$;\\
    For all $A \in \overline\cT^1$:
    \item If $\rk A < r$, then $P(A) = 0$ for all $\rat{P} \in \cP_r$;
    \item If $\rk A \leq r$, then $AP(A)=0$ for all $\rat{P} \in \cP_r$;
    \item If $\rk A = r$ then there is $\rat{P} \in \cP_r$ that is defined at $A$, and $\bim P(A) = \bker A$.
    \end{enumerate}
\end{coro}
\begin{proof}
    By construction~\labelcref{eq:P-Q}, 
    for $I,J \in \binom{[n]}{r}$,
    both $P_{I,J},Q_{I,J} \in \Poly(\cT^1,\cT^1)$
    consist of homogeneous polynomials of degree $r$.
    The rational map $\rat{P}_{I,J}\coloneqq A \mapsto \ffrac{P_{I,J}(A)}{\det(A_{I \times J})}$
    satisfies
    $\Id_n - \rat{P}_{I,J} = A \mapsto
    \ffrac{Q_{I,J}(A)}{\det (A_{I \times J})}$,
    so $\PR(\Id_n - \rat{P}_{I,J}) = \PR(Q_{I,J}) \le r$, since
    $\adj (A_{I \times J}) \in \F^{r \times r}$ and therefore
    \[ \adj (A_{I \times J}) A_{I \times [n]}
    = \sum_{i=1}^r (\adj (A_{I \times J}))_{[r] \times \{i\}} (A_{I \times [n]})_{\{i\} \times [n]}.\]
    The remaining desired properties of $\cP_{r}$ follow immediately from \cref{lemma:P-Q} (applied over $\overline\F$).
\end{proof}

\section{A Motivating Example: The 3-Tensor Case} \label{sec:3tensor}
In this section, we aim to sketch the proof of \cref{main:main} in the case where $d = 2$, i.e., for $3$-tensors, the simplest case in which our work introduces new ideas. We will then discuss our general approach in generalizing to larger $d$, with the goal of motivating some of the technical constructions that come later on. This section is purely expository and does not contain any definitions or results that are used later.

Our starting point is \cref{cor:algrank}, which gives 
formal rational maps
$\rat{R}$ of partition rank at most $r$ such that $\rat{R}(A) \equiv A$ for all $A$ with $\rank A \leq r$. In a sense, this is an ``algebraization'' of the fact that $\PR(A) = \rank(A)$ for all $A$ in $\cT^1$.

We apply the above to $T \in \cT^2$ in the following way. Suppose that $\rank T[v] \leq r$ for all $v \in \overline V$. Then, we know that $\rat{R}(T[v]) \equiv T[v]$, where we view both sides as polynomials in $v$. Consider taking the derivative of both sides with respect to $v$, which we will denote using $\nabla$. In general, the derivative of a map $V \to U$ is a map $V \to U \otimes V^*$, where $V^*$ is the dual space of $V$, and in this case, where $U = \cT^1$, $U \otimes V^*$ can be identified with $\cT^2$. It is not hard to show that if $\rat{F} \equiv h$ then $\nabla \rat{F} \equiv \nabla h$, for a suitable definition of $\nabla \rat{F}$. Therefore we conclude that $\nabla \rat{R}(T[v]) \equiv \nabla T[v]$.

Since $T[v]$ is linear in $v$, $\nabla T[v]$ is simply $T$. Moreover, we claim that for any $v$, $\PR(\nabla \rat{R}(T[v])) \leq 2r$. A formal proof of this fact would take us too far afield, but the general idea is that since we can write $\rat{R}(T[v]) = \sum_i a_i(v) \otimes b_i(v)$, we have
\[\nabla \rat{R}(T[v]) = \sum_i \nabla a_i(v) \otimes b_i(v) + a_i(v) \otimes \nabla b_i(v).\]
Of course, since $\rat{R}$ is a formal rational map, the above derivation was not particularly rigorous, but the general idea is accurate. One additional wrinkle is that one needs to be sure that $\rat{R}$ is actually defined at $T[v]$; by \cref{cor:algrank}, a sufficient condition for this to be true for some $\rat{R}$ is that $\rank T[v] = r$ for some $v \in V$. We conclude the following.
\begin{prop}
If $T \in \cT^2$ and $r \in \NN$ are such that $\rank T[v] \leq r$ for all $v \in \overline V$ and $\rank T[v] = r$ for some $v \in V$, then $\PR(T) \leq 2r$.
\end{prop}
While this result gives us a semblance of how to bound partition rank in terms of other statistics, in general its conditions are far too restrictive. To make it more applicable, we need to make a refinement.

Instead of supposing that $\rank T[v] \leq r$ for all $v \in \overline V$, we instead assume that $\rank T[v] \leq r$ for all $v \in \overline U$, where $U \subseteq V$ is a subspace of low codimension, which we will call $r_2$. Then, by applying the same logic, one can show that the restriction of $T$ to $V \times U$ has partition rank at most $2r$. To extend this to a decomposition of all of $T$, let $P \colon V \to U$ be a projection operator. Then $T(v_1, v_2) = T(v_1, Pv_2) + T(v_1, (\Id-P)v_2)$. The first term has partition rank at most $2r$. On the other hand, the rank of $\Id - P$ is $r_2$, so the second term has partition rank at most $r_2$. Thus $\PR(T) \leq 2r + r_2$. Renaming $r$ to $r_1$, we now have the following.
\begin{prop}
If $T \in \cT^2$, $r_1, r_2 \in \NN$, and $U \subseteq V$ are such that $\dim V/U = r_2$, $\rank T[v] \leq r_1$ for all $v \in \overline U$, and $\rank T[v] = r_1$ for some $v \in U$, then $\PR(T) \leq 2r_1 + r_2$.
\end{prop}
This proposition is sufficiently strong to be useful for general tensors in $\cT^2$; to finish the proof we need to show that all tensors $T$ satisfy this condition for some $U$, $r_1$, and $r_2$, with $r_1, r_2$ bounded by some constant times the analytic rank.

The first condition we will analyze is the condition that $\rank T[v] \leq r_1$ for all $v \in \overline U$. In order to work with this condition, we recall the fact that a matrix has rank at most $r_1$ if and only if all its $(r_1+1) \times (r_1+1)$-minors are zero. Put another way, there exists a set $\cS\subseteq \Poly(\cT^1)$ of polynomials of degree $r_1+1$ such that $\rank T[v] \leq r_1$ for all $v \in \overline U$ if and only if $f(T[v]) = 0$ as polynomials in $v$ for all $f \in \cS$. By the Schwartz-Zippel lemma, to show that $f(T[v]) = 0$ as polynomials in $v$, it suffices to show that $f(T[v]) = 0$ with probability greater than $\frac{r_1 + 1}{\abs{\F}}$, where $v \in U$ is chosen randomly. Thus, the first condition will be satisfied if $\rank T[v] \leq r_1$ for greater than $\abs{U} \frac{r_1 + 1}{\abs{\F}}$ values of $v \in U$. This reduction has the advantage of eliminating $\overline \F$. Moreover, we may now eliminate the condition that $\rank T[v] = r_1$ for some $v$, because if no $v\in U$ satisfies $\rank T[v] = r_1$, we may simply decrease $r_1$ until we do find such a $v$. Thus, we have shown the following.
\begin{prop} \label{prop:r1r2}
If $T \in \cT^2$, $r_1, r_2 \in \NN$, and $U \subseteq V$ are such that $\dim V/U = r_2$ and $\rank T[v] \leq r_1$ for more than
$\frac{r_1+1}{\abs{\F}} \abs{U}$ values of $v \in U$, then $\PR(T) \leq 2r_1 + r_2$.
\end{prop}
Our final step is to make the choice that $U = \ker T[p,-]_2$ for some $p\in V$ (here $-$ represents a dummy element). In this case, $\dim V/U = \rank T[p,-]_2$. Why is this helpful? One can show relatively easily using Jensen's inequality that if $(p, v)$ are selected uniformly at random from $\Z(T)$, then $\E[\rank T[p,-]_2], \E[\rank T[v]] \leq \AR(T)$. Therefore, $\E[2\rank T[v] + \rank T[p,-]_2] \leq 3\AR(T)$. Now, notice that if we condition on a certain value for $p$, the distribution on $v$ is a uniformly random distribution on $\ker T[p, -]_2$. Hence, there must exist some $p\in V$ such that $\E[2\rank T[v] + \rank T[p,-]_2] \leq 3\AR(T)$, where $v \in \ker T[p, -]_2$ is chosen uniformly at random. By Markov's inequality, we must have
\[2\rank T[v] + \rank T[p,-]_2 \leq \frac{3}{1-\eps} \AR(T)\]
with probability at least $\eps$. Now, if we set $r_2 = \rank T[p,-]_2$ and let $r_1 = \floor{\frac{1}{2}(\frac{3}{1-\eps} \AR(T) - r_2)}$ be the greatest integer with $2r_1 + r_2 \leq \frac{3}{1-\eps} \AR(T)$, we find that the conditions of \cref{prop:r1r2} are satisfied for this choice of $r_1$ and $r_2$, provided that $\eps > \frac{r_1 + 1}{\abs{\F}}$. Since any such $r_1$ must satisfy, say, $r_1 \leq \frac{3}{2(1-\eps)} \AR(T)$, we have just proven the following.
\begin{prop}
Let $T \in \cT^2$ and $\e \in (0,1)$.
If $\abs{\F} > \frac{1}{\eps} \big(\frac{3}{2(1-\eps)} \AR(T) + 1\big)$ then $\PR(T) \leq \frac{3}{1-\eps} \AR(T)$.
\end{prop}
Let's take stock of what we have accomplished. With the benefit of hindsight, we can break the previous proof into two main parts, which correspond to \cref{main:LR-AR,main:LR-PR}, respectively:
\begin{itemize}
    \item A probabilistic argument, in which we show that for every $T \in \cT^2$ and $0 < \eps < 1$ we can find $p \in V$, $(r_1,r_2) \in \NN^2$ such that $\rank T[p, -]_2 = r_2$, $\max_{v \in \ker T[p,-]_2} \rank T[v] = \max_{v \in \bker T[p,-]_2} \rank T[v] = r_1$, and $2r_1 + r_2 \leq \frac{3}{1-\eps} \AR(T)$, provided that $\F$ is sufficiently large.
    \item An algebraic argument, showing that if $\rank T[p, -]_2 = r_2$ and $\max_{v \in \ker T[p,-]_2} \rank T[v] = \max_{v \in \bker T[p,-]_2} \rank T[v] = r_1$, then $\PR(T) \leq 2r_1 + r_2$.
\end{itemize}
The key to generalizing to higher $d$ is to iterate the above process. By using \cref{cor:algproj}, it is possible to create a set of formal rational maps that are formulas that perform the same construction as the algebraic argument, but done entirely using algebra. These formulas will depend on the choice of $p$, $r_1$, and $r_2$, and the condition for the formula working will develop into the local rank. Just like how the formulas in \cref{cor:algrank} work for a bigger class of tensors (those with rank at most $r$) than the tensors at which they're defined (those with rank equal to $r$), these formulas hold for a broader class of $T$ as well; we will provisionally call such $T$ \emph{$(p, r_1, r_2)$-valid}.

Then, by applying the same logic as in the $d = 2$ case, it is true that the partition rank of  $T \in \cT^3$ can be bounded in terms of $r_1,r_2,r_3$ if one can find $p_1,p_2,p\in V$ and $r_1,r_2,r_3\in \NN$ such that $\rank T[p_1,p_2,-]_3 = r_3$ and $T[x_3]$ is $(p,r_1,r_2)$-valid for all $x_3 \in \bker T[p_1,p_2,-]_3$, modulo a definedness condition; thus, it remains to show the probabilistic claim that all tensors with low analytic rank satisfy this condition for some $p_1,p_2,p,r_1,r_2,r_3$. To show this, we need two final insights, which we detail below.

We wish to apply the Schwartz-Zippel lemma to derive this absolute condition from discrete data. However, unlike the $d=2$ case, we must apply the Schwartz-Zippel lemma twice, once for many slices $T[x_3]$ of the tensor to show that they are individually $(p, r_1, r_2)$-valid, and then across all slices to show that $T[x_3]$ is $(p,r_1,r_2)$-valid for all $x_3$. (Fortunately this last step is possible since $(p,r_1,r_2)$-validity is also defined by polynomials of bounded degree, in this case $O(r_1r_2)$). The discrete data needed for this argument to work is that for many $x_3 \in \ker T[p_1,p_2,-]_3$, there must be many $x_2 \in \ker T[p,-,x_3]_2$ such that $\rank T[-,x_2,x_3]_1, \rank T[p,-,x_3]_2, \rank T[p_1,p_2,-]_3$ are small. To arrive at this conclusion, we use a combinatorial lemma that derives this ``two-level'' condition from a ``one-level'' condition, namely, that there exist many pairs $(x_2,x_3)$ such that the above conditions hold.

Our final big insight is to interpret these quantities in terms of a random process by setting $p_1 = p$; then checking the above condition reduces to analyzing quantities of the form $\rank T[p_1,p_2,-]_3$, $\rank T[p_1,-,x_3]_2$, and $\rank T[-,x_2,x_3]_1$, where $x_3 \in \ker T[p_1,p_2,-]_3$ and $x_2 \in \ker T[p_1,-,x_3]_2$. Now, let $(p_1,p_2,x_3)\in \Z(T)$ be uniformly random, and then choose $x_2 \in \ker T[p_1,-,x_3]_2$ uniformly at random. Under this choice, it turns out that $(p_1,x_2,x_3) \in \Z(T)$ is also a uniformly random variable, implying that all three of the ranks above can be bounded in expectation.

At this point, we have uncovered all major ideas in our proof. In summary, the proof consists of
\begin{itemize}
\item an algebraic argument (\cref{main:LR-PR}) which inductively constructs formulas with low partition rank through
the use of derivatives;
\item a condition (the local rank) required for such formulas to work, depending on ranks of the form $\rank T[p_1,\ldots,p_{i-1},-,x_{i+1},\ldots,x_d]_i$ for some $\vx, \vp \in \overline V^d$; and
\item a probabilistic argument (\cref{main:LR-AR}) showing that such a condition holds, consisting of:
\begin{itemize}
\item a random process on $\Z(T)$ (\cref{def:uniform-resample}),
\item bounding the above ranks in expectation (\cref{prop:exp-rk-bound}),
\item translating a ``one-level'' probabilistic condition to a ``multi-level'' condition (\cref{lem:pruning} and \cref{lemma:lexMarkov}), and
\item applying the Schwartz-Zippel lemma many times (\cref{claim:R-LR-bound}).
\end{itemize}
\end{itemize}

\section{Local Rank} \label{sec:lr}
The goal of this section is to define and characterize two closely related notions of rank, which we call the local rank and the algebraic local rank, for tensors $T \in \cT^d$, 
both of which reduce to the standard definition of matrix rank when $d = 1$. 
There are two notable features of our notions of rank. 
First, the local rank of a tensor depends on a point $\vp \in V^d$ and in some sense describes the behavior of $T$ ``near'' the point $\vp$. 
Second, the local rank of a tensor $T \in \cT^d$ is not one number in $\NN$ but a tuple in $\NN^d$.

\subsection{Definition of (algebraic) local rank} 
Before we define local rank, we first make some preliminary definitions.
We henceforth write $\prec$ for the colexicographic order on tuples, i.e.\ for $\va,\vb \in \NN^d$, we have  $\va \prec \vb$ if and only if there exists an $i \in [d]$ with $a_i < b_i$ and $a_j=b_j$ for all $j > i$.
We write $\va \preceq \vb$ if either $\va \prec \vb$ or $\va = \vb$.
Note that $\va \preceq \vb$ is \emph{weaker} than $a_i \le b_i$ for all $i$.

Let us recall the notation $\vx \ominus_i \vy \coloneqq (x_1,\ldots,x_i,y_{i+1},\ldots,y_d)$.
The formal definition of adjoint pairs and local rank follows.

\begin{definition}[adjoint] \label{def:adjoint}
For $T \in \cT^d$, we call $(\vp,\vx) \in V^d \times V^d$ an \emph{adjoint pair} of $T$, denoted $\vp \dashv_T \vx$, if $T(\vp \ominus_i \vx) = 0$ for all $i \in [d-1]$.
\end{definition}
That is, $\vp \dashv_T \vx$ whenever the $d-1$ points
\begin{center}
    $(p_1,p_2,\ldots,p_{d-1},x_d) \in \Z(T)$\\
    $(p_1,p_2,\ldots,x_{d-1},x_d) \in \Z(T)$\\
    $\vdots$\\
    $(p_1,x_2,\ldots,x_{d-1},x_d) \in \Z(T)$
\end{center}
all lie in the zero set of $T$.
Note that \cref{def:adjoint} is independent of $p_d$ and of $x_1$.
Let us recall the notation $T[\vx]_i=T(x_1,\ldots,x_{i-1},-,x_{i+1},\ldots,x_d) \in \cT^1$.

\begin{definition}[local rank]\label{def:LR}
Let $T \in \cT^d$ and $\vp \in V^d$. The \emph{local rank of $T$ at $\vp$} is
    \[\LR_\vp(T) \coloneqq \maxprec_{\vx \in V^{d}\colon \vp \dashv_T \vx} \paren*{
 \rank T[\vp \ominus_1 \vx]_1, \ldots, \rank T[\vp \ominus_d \vx]_d}\]
    where $\max^\prec$ denotes the maximum taken under the order $\prec$. 
    The \emph{algebraic local rank of $T$ at $\vp$} is the local rank of $T$ at $\vp$ when viewing $T$ as a tensor over $\overline \F$.
\end{definition}
Note that $\bLR_\vp$ differs from $\LR_\vp$ in that its maximization is over the infinite vector space $\overline V^{d}$. Trivially, $\LR_\vp(T) \preceq \bLR_\vp(T)$.
It is also worth noting that $\LR_\vp$ is independent of $p_d$ and of $x_1$, that its last coordinate is $\rk T[\vp]_d = \rk T(p_1,\ldots,p_{d-1},-)$ (so does not involve $\vx$), and that its first coordinate is equal to $\rk T[\vx]_1 = \rk T(-,x_2,\ldots,x_d)$ (so does not involve $\vp$).

One important way to conceptualize the definitions of adjointness and local rank is to consider $\vp$ and $\vx$ as representing a walk through $\Z(T)$, where we start at a point $\vp \ominus_{d-1} \vx$, 
and for each $i$ descending from $d-1$ to $1$, we replace $p_i$ with $x_i$, leading to the point $\vp \ominus_{i} \vx$, which must remain in $\Z(T)$. 
Observe that if $\vp \ominus_{i} \vx$ is fixed, then the set of $x_i \in V$ for which $\vp \ominus_{i-1} \vx \in \Z(T)$ is simply $\ker T[\vp \ominus_i \vx]_i$. Thus, the ranks in the definition of local rank represent the ``amount of choice'' we have at any stage in this process. This interpretation will be crucial in \cref{sec:LR-AR}.

We finish this section by giving a recursive definition of local rank.
\begin{prop}\label{prop:lrdefrecur}
Let $d > 1$, $\vp \in V^d$, $T \in \cT^d$, and $\vr\in \NN^d$. Then
\[\LR_{\vec p}(T) = \paren*{\maxprec_{x \in \ker T[\vec p]_d} \LR_{\vec p'}(T[x]), \rank T[\vp]_d}.\]
\end{prop}
\begin{proof}
Note that if $\vx \in V^d$, for each $i \in [d-1]$ we have $T[\vec p \ominus_i \vec x]_i = T[x_d][\vec p' \ominus_i \vec x']_i$. Similarly, we have $T(\vec p \ominus_i \vec x) = T[x_d](\vec p' \ominus_i \vec x')$, so $\vec p \dashv_T \vec x$ if and only if $\vec p' \dashv_{T[x_d]} \vec x'$ and $T[x_d](\vec p') = 0$. The latter condition is equivalent to $x_d \in \ker T[\vec p]_d$. Hence we have
\begin{align*}
\LR_{\vec p}(T) &= \maxprec_{\vp \dashv_T \vx} \,\,\, \big(\rank T[\vp \ominus_1 \vx]_1, \ldots, \rank T[\vp \ominus_d \vx]_d \big) \\
 &= \maxprec_{\substack{x_d \in \ker T[\vec p]_d \\ \vec p' \dashv_{T[x_d]} \vec x'}} \paren*{
 \rank T[x_d][\vp' \ominus_1 \vx']_1, \ldots, \rank T[x_d][\vp' \ominus_d \vx']_{d-1}, \rank T[\vp]_d} \\
 &= \paren*{\maxprec_{x_d \in \ker T[\vec p]_d} \LR_{\vec p'}(T[x_d]), \rank T[\vp]_d},
\end{align*}
as desired.
\end{proof}

\subsection{Examples}
Let us give a few simple examples of local rank. 
As before, we let a dash (i.e., $-$) represent a ``dummy element'' of $V$, which will be useful as $T[\vx]_i$ is independent of $x_i$.
Let us also recall the notation $T[v]=T(-,\ldots,-,v) \in \cT^{d-1}$.

\begin{example}
If $d = 1$, then adjointness is a vacuous condition, so the local rank and the algebraic local rank are both equal to matrix rank.
\end{example}

\begin{example}
\label{ex:lr2}
For $T \in \cT^2$, 
we have $\vp \dashv_T \vx$ if $(p_1,x_2) \in \Z(T)$. So, the definition of local rank reads
\begin{align*}
  \LR_\vp(T) &= \maxprec_{(p_1,x_2) \in \Z(T)} (\rank T[p_1,x_2]_1, \rank T[p_1,p_2]_2)
    \\ &= \maxprec_{(p_1,x_2) \in \Z(T)} (\rank T(-,x_2), \rank T(p_1,-))\\
    &= \big(\max_{x \in \ker T(p,-)} \rank T(-,x),\, \rank T(p,-)\big)
    \\&= \big(\max_{x \in U} \rank T[x],\, \codim U \big)
\end{align*}
where $U = \ker T(p,-)$.
In particular, $\LR_{\vec{0}}(T) = (\max_{x \in V} \rank T[x], 0)$, which is effectively the maximum rank of a space of matrices. Over an infinite field, this is 
\emph{commutative rank} \cite{FortinRe04,Wigderson17} of the matrix space $\setmid{T[x]}{x\in V}$. 
Since the commutative rank of a space of matrices is not necessarily the maximum rank, this provides an example of an instance when the local rank and the algebraic local rank can differ.
\end{example}

The identity tensor is known to have partition rank and analytic rank growing with the number of variables (e.g.,~\cite{Lovett19,Naslund20}). 
A similar phenomenon happens for local rank $\LR_\vp$, though of course this depends on the choice of the point $\vp$.

\begin{example}
    Let $d \in \NN$, $V=\F^n$, and $\Id^{(d)}_n \in \cT^d$ be the ``identity tensor''\footnote{Equivalently, the multilinear polynomial $\sum_{i=1}^n \prod_{j=1}^k x_{j,i}$ of degree $k=d+1$ (sum of variable-disjoint monomials).} given by
    \[\Id^{(d)}_n(x_1,\ldots,x_d) = \Big( \prod_{j=1}^d x_{j,1},\ldots, \prod_{j=1}^d x_{j,n} \Big).\]
    For $\vp \in V^d$, let $r$ be the number of coordinates $i \in [n]$ such that $p_{1,i},\ldots,p_{(d-1),i}$ are all nonzero.
    We have $\Id^{(d)}_n[\vp]_d = x_d \mapsto \big( x_{d,1}\prod_{j=1}^{d-1} p_{j,1},\ldots, x_{d,n}\prod_{j=1}^{d-1} p_{j,n} \big)$, and so $\rk \Id^{(d)}_n[\vp]_d = r$.
    In particular, for a ``generic'' $\vp \in V^d$, $\LR_\vp(\Id^{(d)}_n)$ has a coordinate (its last, which is $r$) that grows with $n$.
\end{example}

It is possible to modify the definition of local rank to use any total order on $\NN^d$, and so far we have not used any property of $\prec$ beyond the fact that it extends the product (i.e., coordinatewise) partial order on $\NN^d$.
The next example shows that the choice of order can have a dramatic effect. By \cref{ex:lr2}, such an example would have to be of a tensor in $\cT^d$ for $d \ge 3$.
\begin{example} \label{ex:lr3}
Let $n \in \NN$, $V = \F^{n+1}$, and consider the tensor $T \in \cT^3(V,\F)$ given by
\[T(x,y,z) = (x_1y_1z_1 + x_2y_1z_2,\, x_3y_1z_2,\,\ldots,\, x_{n+1}y_1z_2,\,0).\]
Let $\vp = \big(e_1,0,-)$, where $e_1$ is the first standard unit vector.
For $\vx=(a,b,c) \in V^3$, we have $\vp \adjoint{T} \vx$ if and only if
$T(e_1,0,c)=0$ and $T(e_1,b,c)=0$,
which is equivalent to $(b_1c_1, 0,\ldots,0)=0$, or simply $b_1c_1=0$.
Moreover, for such an $\vx$ we have
\begin{align*}
    T[\vp \ominus_1 \vx]_1 &= \paren[\big]{v \mapsto b_1c_2(v_2, v_3,\ldots, v_{n+1},0)} , \\
T[\vp \ominus_2 \vx]_2 &= \paren[\big]{v \mapsto (c_1v_1,0,\ldots,0)} , \\
T[\vp \ominus_3 \vx]_3 &= \paren[\big]{v \mapsto (0,0,\ldots,0)} .
\end{align*}
For the first two matrix slices, this means that
\[\setmid[\big]{ (\rk T[\vp \ominus_1 \vx]_1, \rk T[\vp \ominus_2 \vx]_2) }{ \vp \adjoint{T} \vx } 
= \set{ (0,0), (0,1), (n,0) } .\]
Although $n$ can be arbitrarily large, local rank maximizes based on $\prec$, so
$\LR_\vp(T) = (0,1,0)$.
\end{example}

\subsection{Algebraic definition of algebraic local rank}\label{subsec:reduction-op}

Our goal is to prove that algebraic local rank can be defined by low-degree polynomial equations.
Towards this goal, we use rational maps that parameterize (or surjectively project onto) the kernel of a matrix---given in \cref{def:Proj}.
This will be extended to a parameterization of adjoint pairs, leading to an algebraic definition of algebraic local rank (\cref{lem:lralgrecur} below), and then to the goal stated above (\cref{prop:ddeg} below).

It will be useful to first define the following reduction map, which is used to parameterize 
the tensor slices $T[x]$ as $x$ varies on a subspace.

\begin{definition}[Tensor reduction operator]\label{def:tensor-proj}
For $\vp\in V^d$ and $P \in \Poly(\cT^1, \cT^1)$, 
define the polynomial map $\Red^P_\vp \in \Poly(\cT^d \times V, \cT^{d-1})$ as 
\[\Red^P_{\vp}(T, v) \coloneqq T[P(T[\vp]_d)v].\]
\end{definition}

\begin{example}
    If $T[\vp]_d$ is diagonal, say $\diag(d_1,\ldots,d_k,0,\ldots,0)$, then for some $\rat{P} \in \cP_k$ (recall \cref{def:Proj} and \cref{ex:diagonal-proj}) 
    we have
    $P(T[\vp]_d) = c\begin{psmallmatrix} 0 & 0\\0 & I_{n-k} \end{psmallmatrix}$ for $c=d_1\cdots d_k$,
    and therefore $\Red^P_{\vp}(T, v) = T[v']$ $(\in \cT^{d-1})$ where $v'=c(0,\ldots,0,v_{k+1},\ldots,v_n)$.
\end{example}

Next we give an algebraic analogue of \cref{prop:lrdefrecur},
using the family of projection maps $\cP_{r} \sub \FRat(\cT^1, \cT^1)$ (\cref{def:Proj}).
We continue the convention that for $\rat{P} \in \cP_{r}$, the italicized $P$ denotes its numerator (a polynomial map $P \in \Poly(\cT^1, \cT^1)$).
Let us also recall that statement of \cref{cor:algproj}.
\begin{quote}
    \begin{enumerate}
    \item For all $\rat{P} \in \cP_r$, $\PR(\Id_n-\rat{P}) \le r$ and $P$ is homogeneous of degree $r$;\\
    For all $A \in \overline\cT^1$:
    \item If $\rk A < r$, then $P(A) = 0$ for all $\rat{P} \in \cP_r$;
    \item If $\rk A \leq r$, then $AP(A)=0$ for all $\rat{P} \in \cP_r$;
    \item If $\rk A = r$ then there is $\rat{P} \in \cP_r$ that is defined at $A$, and $\bim P(A) = \bker A$.
    \end{enumerate}
\end{quote}

\begin{lemma}\label{lem:lralgrecur}
    For every $T \in \overline\cT^d$, $\vp \in  V^d$, and $\vr\in \NN^d$,
    we have that $\bLR_\vp(T) \preceq \vr$ if and only if $\rank T[\vp]_d \leq r_d$ and $\bLR_{\vp'}(\Red^P_{\vp}(T,v)) \preceq \vr'$ for all 
    $\rat{P} \in \cP_{r_d}$ and $v \in \overline V$.
\end{lemma}
\begin{proof}
    We first claim that
    \begin{equation}\label{eq:alg-LR}
    \bLR_\vp(T) = \Big(\maxprec_{P,v} \,\,\,\,\, \bLR_{\vp'}(\Red^P_\vp(T,v)),\, \rank T[\vp]_d \Big)
    \end{equation}
    where the maximum is over all $\rat{P} \in \cP_{\rk T[\vp]_d}$ and $v \in \overline V$.
    By \cref{cor:algproj}, we have $\bim P(A) \sub \bker A$, and moreover, there is $\rat{P} \in  \cP_{r}$ such that $\bim P(A)=\bker A$.
    Together with \cref{prop:lrdefrecur}, this implies~\labelcref{eq:alg-LR}.

    First, assume $\rank T[\vp]_d > r_d$;
    then $\bLR_\vp(T) \succ \vr$, so we are done.
    Next, assume $\rank T[\vp]_d < r_d$. Then $\bLR_\vp(T) \prec \vr$, so we need to prove $\bLR_{\vp'}(\Red^P_{\vp}(T,v)) \preceq \vr'$ for $\rat{P} \in \cP_{r_d}$, $v \in \overline V$.
    By \cref{cor:algproj}, $P(T[\vp]_d) = 0$, which implies $\Red^P_{\vp}(T,v)=T[0]=0  \in \overline\cT^{d-1}$.
    Thus, $\bLR_{\vp'}(\Red^P_{\vp}(T,v)) = \vec{0} \preceq \vr'$, and we are done again.
    
    Finally, assume $\rank T[\vp]_d = r_d$. Then $\bLR_\vp(T) \preceq \vr$ if and only if $\bLR_\vp(T)' \preceq \vr'$. 
    By~\labelcref{eq:alg-LR}, this is equivalent to $\bLR_{\vp'}(\Red^P_\vp(T,v)) \preceq \vr'$ for all $\rat{P} \in \cP_{r_d}$ and $v \in \overline V$, 
    as desired. 
    This completes the proof.
\end{proof}

\subsection{Rigidity of algebraic local rank}\label{subsec:rigid}

We can now prove that the set of tensors $\{T \in \overline\cT^d \mid \bLR_{\vec p}(T) \preceq \vec r\}$ is in fact Zariski closed, and moreover, that it has low ``defining degree'': it can be cut out by polynomials of degree at most 
$\prod_{i=1}^d (r_i + 1)$.
The proof follows by combining the recursive algebraic definition of local rank in \cref{lem:lralgrecur} together with the projection maps in \cref{def:Proj}.

\begin{prop} \label{prop:ddeg}
Let $d \geq 1$, $\vp \in  V^d$, and $\vr\in \NN^d$. There exists a set $\cS_{\vp, \vr} \sub \Poly(\cT^d)$ such that for every $T \in \overline \cT^d$, we have $\bLR_{\vec p}(T) \preceq \vec r$ if and only if $f(T) = 0$ for all $f \in \cS_{\vec p, \vec r}$. 
Moreover, the polynomials in $\cS_{\vp, \vr}$ have degree at most $(r_1 + 1)(r_2 + 1) \cdots (r_d + 1)$.
\end{prop}

\begin{proof}
We use induction on $d$. If $d = 1$, we can just let $\cS_{\vec p, \vec r}$ be the $(r_1 + 1) \times (r_1 + 1)$-minors of $T$.

Now let $d > 1$ and assume that the statement is true for $d - 1$. 
We take $\cS_{\vp, \vr} \sub \Poly(\cT^d)$ to be the set consisting of the following polynomials:
\begin{itemize}
    \item The $(r_d + 1) \times (r_d + 1)$-minors of $T[\vec p]_d$. 
    \item For every $f \in \cS_{\vec p', \vec r'}$ and $\rat{P} \in \cal{P}_{r_{d}}$, 
    the coefficients of $f(\Red^P_\vp(T, v))$ when viewed as a polynomial in $v$.
\end{itemize}
Let $V_1$ and $V_2$ be the zero sets cut out by the two families of polynomials above, respectively.
Observe that $\deg_T \Red^P_\vp(T, v)$, the degree of $\Red^P_\vp(T, v)$ in $T$, is $r_d+1$;
indeed, this follows since $T[\vec p]_d$ is linear in $T$, and since $P$ is homogeneous of degree $r_d$ by \cref{cor:algproj}.

First we verify that the degree condition is satisfied. 
Since $T[\vec p]_d$ is linear in $T$, the polynomials cutting out $V_1$ all have degree $r_d + 1$. 
Furthermore, the polynomials cutting out $V_2$ all have degree at most $\prod_{i=1}^d (r_i+1)$
since
\[\deg_T f(\Red^P_\vp(T, v)) \le \deg f \cdot \deg_T \Red^P_\vp(T, v) \le (r_1 + 1) \cdots (r_{d-1} + 1) \cdot (r_d+1),\]
where the last inequality uses the induction hypothesis.

It remains to show that the polynomials in $\cS_{\vec p, \vec r}$ do determine local rank, that is, for every $T \in \overline\cT^d$, $\bLR_{\vec p}(T) \preceq \vec r$ if and only if $T \in V_1 \cap V_2$. 
Note that $T \in V_1$ if and only if $\rank T[\vp]_d \leq r_d$.
Moreover, $T \in V_2$ if and only if for every $f \in \cS_{\vec p', \vec r'}$ and $\rat{P} \in \cal{P}_{r_{d}}$, $f(\Red^P_\vp(T, -))$ is the zero polynomial,
which, by the infinitude of $\overline \F$, holds if and only if $f(\Red^P_\vp(T, v)) = 0$ for every $v \in \overline V$.
Thus, by the induction hypothesis, $T \in V_2$ if and only if $\bLR_{\vp'}(\Red^P_\vp(T, v)) \preceq \vr'$ for every $v \in \overline V$ and $\rat{P} \in \cal{P}_{r_{d}}$.
Therefore, by \cref{lem:lralgrecur}, $T \in V_1 \cap V_2$ if and only if $\bLR_{\vec p}(T) \preceq \vec r$, which completes the proof.
\end{proof}

\begin{remark}\label{remark:lex-Zariski}
The lexicographic nature of the order $\prec$ is necessary for a result like \cref{prop:ddeg} to hold.
For example, we claim that the Zariski closure of $X\coloneqq\{ T \in \overline\cT^2 \mid \bLR_\vp(T) = (0,1) \}$, for some $\vp \in V^2$, contains a $T'$ with $\bLR_\vp(T')=(n,0)$, for all $n$.
This strongly suggests that it is impossible for the set of tensors of local rank ``at most'' $\vr$ to be defined by polynomial equations if any non-lexicographic order is used; on the other hand, we know from \cref{prop:ddeg} that the set $\{ T \in \overline\cT^2 \mid \bLR_\vp(T) \preceq (0,1) \}$ is Zariski closed, and $T'$ is indeed a member since $\bLR_\vp(T')=(n,0) \preceq (0,1)$.\footnote{Lest the reader raise the objection that the definition of local rank itself uses the order $\prec$, we note that for the case of elements of $\cT^2$, the lexicographic nature of $\prec$ does not show up at all in the definition of the local rank (see \cref{ex:lr2}). Thus, in some sense the use of the order $\prec$ is forced.}
To show the claim above, consider the tensor $T \in \cT^3(V,\F)$ with $V=\F^{n+1}$ from \cref{ex:lr3}, and note that, by the same argument there, we have for $\vp = (e_1,-)$ 
and $c \in \overline{V}$ that $\bLR_\vp(T[c])$
is $(0,1)$ when $c_1 \neq 0$, and is $(n,0)$ when $c_1 = 0$ and $c_2 \neq 0$.
Thus, for the linear subspace $Y\coloneqq T\big[\overline{V}\big] \sub \overline\cT^2$, the Zariski closure of $X \cap Y$ in $Y$, and thus the Zariski closure of $X$, contains a tensor $T'$ with $\bLR_\vp(T') = (n,0)$.
\end{remark}

\cref{prop:ddeg} can be used to deduce a rigidity result for the local rank of the slices of a tensor,
which we will need in \cref{sec:LR-AR}.
We recall (a special case of) the Schwartz-Zippel lemma. \begin{lemma}[\cite{Ore22}\footnote{Technically, Ore's original paper only states this result for prime fields, i.e., $\F = \F_p$ for some prime $p$. Nonetheless, the proof given there works across all finite fields.}] \label{lemma:SZ}
Let $\F$ be a finite field and $W$ be a finite-dimensional vector space over $\F$. 
For every $f \in \Poly(W)$, if $\abs{\set{ x \in W \mid f(x) = 0}}/\abs{W} > \deg(f)/\abs{\F}$ then $f=0$.
\end{lemma}
\begin{coro}\label{coro:LR-SZ}
Let $\F$ be a finite field, $d \geq 1$, $T \in \cT^{d+1}$, $\vp \in V^d$, $\vr \in \NN^d$, and $W \sub V$ an $\F$-linear subspace. If $\Pr\brac*{\bLR_\vp(T[x]) \preceq \vr} \ge \d$ for a uniformly random $x \in W$,
then $\bLR_\vp(T[x]) \preceq \vr$ for all $x \in \overline{W}$,
provided $\d\abs{\F} > \prod_{i=1}^d (r_i+1)$.
\end{coro}
\begin{proof}
We claim that every $f \in \cS_{\vp,\vr}$ from \cref{prop:ddeg} satisfies $f(T[x]) = 0$ for all $x \in \overline W$.
Indeed, let $g \in \Poly(W)$ be given by $g(x) = f(T[x])$, and note that $\Pr[g(x)=0] \geq \delta$ for a uniformly random $x \in W$;
since $\deg(g) = \deg(f) \le \prod_{i=1}^d(r_i+1) < \delta\abs{\F}$, we have by \cref{lemma:SZ} that $g = 0$, that is, $f(T[x]) = 0$ for all $x \in \overline W$.
We deduce that $\bLR_\vp(T[x]) \preceq \vr$ for all $x \in \overline W$. 
\end{proof}

\section{Local Rank is Bounded by Analytic Rank}\label{sec:LR-AR}
In this section, $\F$ is always finite. Given a vector $\vr \in \NN^d$, we denote 
\[\norm{\vr} = \sum_{i=1}^d 2^{d-i} r_i.\]
Our goal is to prove \cref{main:LR-AR}.
Our proof plan follows these steps:
\begin{itemize}
    \item In \cref{subsec:uniform-resample} we define a certain sampling process on subsets of product spaces with useful uniformity conditions.
    \item In \cref{subsec:exp-rk-bound} we prove that the ranks of each $T[\vp]_i$ can be bounded in expectation by $\AR(T)$ if $\vp$ is chosen uniformly at random from $\Z(T)$.
    \item In \cref{subsec:prob-lex} we define a probabilistic version of the order $\preceq$ for random tuples, denoted $\preceq_\delta$, and show how it follows from some simple conditions.
    \item In \cref{subsec:problr} we relate $\preceq_\delta$ to local rank, using \cref{coro:LR-SZ}. 
    \item In \cref{subsec:rank-vec} we combine all of the above to prove \cref{main:LR-AR}: we use resampling to define ``probabilistic'' local rank,
    bound its norm in expectation, and convert this bound to a bound on local rank.
\end{itemize}
Much of the complexity of this section stems from the subtlety of translating between bounds on norms of vector and bounds on the vectors themselves. 
The main results are:
\begin{itemize}
    \item \cref{prop:exp-rk-bound}, used to bound the norm of probabilistic local rank by analytic rank, 
    \item \cref{lemma:lexMarkov}, used to convert the preceding bound to a bound on the probabilistic local rank vector itself (ordered by $\preceq_\delta$)
    \item \cref{claim:R-LR-bound}, used to convert the preceding bound to a (lexicographic) bound on local rank.
\end{itemize}

The reader who wishes to skim this section is advised to skip directly to \cref{subsec:rank-vec} to get a sense of the overall argument.

\subsection{Uniform resampling}\label{subsec:uniform-resample}
We first record the following easy but useful observation about conditional uniform distributions.

\begin{fact}\label{fact:uniform-condition}
    If $\vec{U}$ is a uniformly random element of a finite set $Z \sub \prod_{i=1}^d \Omega_i$
    then, for every $i\in[d]$ and $\vx \in \prod_i \Omega_i$, $\vec{U}$ conditioned on $U_j=x_j$ for all $j \neq i$
    is a uniformly random element of 
    \[\ext_Z^i(\vx) \coloneqq \setmid{\omega \in \Omega_i }{(x_1,\ldots,x_{i-1},\omega,x_{i+1},\ldots,x_n) \in Z}.\]
\end{fact}

\cref{fact:uniform-condition} motivates the following definition.
We recall the notation $\vx \ominus_i \vy = (x_1,\ldots,x_i,y_{i+1},\ldots,y_d)$.

\begin{definition}\label{def:uniform-resample}
    A \emph{uniform resample} from a finite subset $Z \subseteq \Omega \coloneqq \prod_{i=1}^d \Omega_i$ is a random pair $(\vU,\vX) \in Z \times Z$ obtained as follows:
    \begin{itemize}
        \item $\vU$ is chosen uniformly at random from $Z$, and
        \item $X_i$, for each $i=d,\ldots,1$, is chosen uniformly at random from
        $\ext_Z^i(\vU \ominus_{i} \vX)$.\footnote{That is, uniformly at random from 
        $\setmid{x \in \Omega_i}{(U_1,\ldots,U_{i-1},x,X_{i+1},\ldots,X_d) \in Z}$.}
    \end{itemize}    
\end{definition}

The easy induction argument below shows that a uniform resampling from $Z$ yields 
a sequence of uniformly random elements of $Z$.
We recall the notation $\vx \odot_i v = (x_1,\ldots,x_{i-1},v,x_{i+1},\ldots,x_d)$.

\begin{claim}\label{claim:uniform-resample}
    If $(\vU,\vX)$ is a uniform resample from $Z$ then
    for each $0 \le i \le d$, $\vU \ominus_i \vX$ 
    is a uniformly random element of $Z$.
\end{claim}
\begin{proof}
    Put $\vZ^i = \vU \ominus_i \vX$.
    We use downward induction on $i$. The base case of $i=d$ follows from definition as $\vZ^d = \vU$.
    
    Fix $c \in \Omega$.\footnote{The use of $c$ is for notational convenience only.}
    Let $\vC^i = (U_1,\ldots,U_{i-1},c,X_{i+1},\ldots,X_n)$,
    so that $\vZ^{i}=\vC^i \odot_i U_i$ and $\vZ^{i-1}=\vC^i \odot_i X_i$.
    For the induction step, assume that $\vZ^{i}$ is a uniformly random element of $Z$; we aim to prove the same for $\vZ^{i-1}$.
    Conditioned on $\vC^i$, 
    $X_i$ has the same distribution as $U_i$, by \cref{fact:uniform-condition}.
    For every $\vz\in Z$, write $\vz^i = (z_1,\ldots,z_{i-1},c,z_{i+1},\ldots,z_n)$, and note that
    \begin{align*}
      \Pr[\vZ^{i-1} = \vz] &= \Pr[\vC^i=\vz^i \,\land\, X_i = z_i]
      = \Pr[\vC^i=\vz^i]\cdot\Pr[X_i = z_i \mid \vC^i=\vz^i]\\
      &= \Pr[\vC^i=\vz^i]\cdot\Pr[U_i = z_i \mid \vC^i=\vz^i]
      =  \Pr[\vC^i=\vz^i \,\land\, U_i = z_i]
     = \Pr[\vZ^i = \vz] .
    \end{align*}
    This completes the proof.
\end{proof}

Recall the notation $T[\vx]_i=T(x_1,\ldots,x_{i-1},-,x_{i+1},\ldots,x_d) \in \cT^1$.
We note that for the zero set of a tensor $Z=\Z(T)$, we have $\ext_{\Z(T)}^i(\vx)=\ker T[\vx]_i$.
Put differently, for $T \in \cT^d$, $\vx \in V^d$, and $i \in [d]$, we have
\begin{equation}\label{eq:ext-ZT}
    y \in \ker T[\vx]_i \iff \vx \odot_i y \in \Z(T) .
\end{equation}

\subsection{Analytic rank bounds expected ranks}\label{subsec:exp-rk-bound}
As we next show, the analytic rank bounds the expected rank of the matrices $T[\vz]_i$ when $\vz$ is chosen uniformly at random from $\Z(T)$.

\begin{prop}\label{prop:exp-rk-bound}
Let $\F$ be a finite field and $d \geq 1$.
For every $T \in \cT^d$ and $i \in [d]$,
\[\E_{\vz}[\rank T[\vz]_i] \leq \AR(T)\]
where $\vz$ is chosen uniformly at random from $\Z(T)$.
\end{prop}
\begin{proof}
Let 
$\cV_i = \setmid{\vx\in V^d}{x_i = 0}$ 
and $q = \abs{\F}$.
By~\labelcref{eq:ext-ZT}, $\Z(T)$ decomposes as
\[\Z(T) = \{\vx \odot_i y \mid \vx \in \cV_i,\, y \in \ker T[\vx]_i\}.\]
We deduce the equality
\[\abs{\Z(T)} = \sum_{\vx \in \cV_i} \, \sum_{y \in \ker T[\vx]_i} 1 = \sum_{\vx \in \cV_i} \abs{\ker T[\vx]_i},\]
as well as the equality
\[\sum_{\vz \in \Z(T)} \rank T[\vz]_i
= \sum_{\vx \in \cV_i} \, \sum_{y \in \ker T[\vx]_i} \rank T[\vx]_i =
\sum_{\vx \in \cV_i} \abs{\ker T[\vx]_i} \rank T[\vx]_i.\]

Now, let $\vX$ be chosen uniformly at random from $\cV_i$, and write $X = \abs{\ker T[\vX]_i}/\abs{V}$. Then $\abs{\Z(T)}/\abs{V}^d = \E[X]$, and $\sum_{\vz \in \Z(T)} \rank T[\vz]_i/\abs{V}^d = \E[-X\log_q X]$ since  $X=q^{-\rank T[\vX]_i}$.
We deduce that
\begin{multline*}
    \E_{\vz}[\rank T[\vz]_i] = \frac{\sum_{\vz \in \Z(T)} \rank T[\vz]_i}{\abs{\Z(T)}}
    = \frac{\E[-X\log_q X]}{\E [X]} \\
    \leq \frac{-\E[X]\log_q\E[X]}{\E [X]}
    = -\log_q\E[X] 
    = -\log_q \frac{\abs{\Z(T)}}{\abs{V}^d} 
    = \AR(T),
\end{multline*}
where the inequality follows from Jensen's inequality applied to the function $t \mapsto -t\log_q t$, which is concave for $t \geq 0$
(as $(t\log t)'' = 1/t \ge 0$). This completes the proof.
\end{proof}

\begin{remark}
    Interestingly, if in \cref{prop:exp-rk-bound} we change the random choice to be uniformly at random from all of $V^d$, rather than from $\Z(T)$, then one can prove that the reverse inequality holds.
    (Indeed, in this case we have convexity rather than concavity; 
    in the notation of the proof, 
    $\E_{\vx \in V^d} [\rk T[\vx]_i] = \E_{\vx \in \cV_i} [\rank T[\vx]_i] = \E [-\log_q X] \ge -\log_q \E[X] = \AR(T)$.)
    In fact, the reverse inequality can hold with an arbitrarily large gap.
    For example, consider the $2$-linear map $T(x,y)=(x_1y_1,\ldots,x_1y_n)\colon V^2 \to V$, where $V=\F_q^n$.
    The linear map $T[x]_2 \colon V \to V$, given by $T[x]_2 \colon y \mapsto x_1y$, is invertible if $x_1\neq 0$, and is the zero map if $x_1=0$.
    Thus, $\E_{\vx \in V^2}[\rank T[\vx]_2] = \E_{x \in V}[\rank T[x]_2]
     = (1-1/q)n \approx n$, whereas $\AR(T) \approx 1$, as $\abs{\Z(T)}/\abs{V}^2 = (q^{n-1}\cdot q^{n} + (q^n-q^{n-1}))/q^{2n}$.
    Furthermore, one cannot replace the expectation with maximization in \cref{prop:exp-rk-bound} (with an arbitrarily large gap); 
    indeed, for $T$ as above,
    $\max_{\vz \in \Z(T)} \rank T[\vz]_2 \ge \rank T[(e_1,0)]_2 = \rank T[e_1]_2 = n$.
\end{remark}

\subsection{Pruning lemmas}\label{subsec:prob-lex}
In this section, if $\vx = (x_1,\ldots,x_d)$ and $0 \leq i \leq d$, then we let $\vx_{\leq i} = (x_1,\ldots,x_i)$ and $\vx_{\geq i} = (x_i,\ldots,x_d)$. 
(So, for example, $\vx'=(x_1,\ldots,x_{d-1}) = \vx_{\le d-1}$.)
Analogously, for a subset of a product set $T \sub \prod_{i=1}^d \Omega_i$ and $0 \leq i \leq d$, 
we denote by $T_{\le i}$ the image of $T$ under the projection onto $\prod_{j=1}^i \Omega_j$,
and by $T_{\ge i}$ the image of $T$ under the projection onto $\prod_{j=i}^d \Omega_j$.

We need the following probabilistic lemma which, roughly speaking, guarantees that \emph{every} prefix (or projection) has a large number of extensions (or preimages), 
and which may be of independent interest.

\begin{lemma}[pruning] \label{lem:pruning}
    Let $\vX$ be a random vector taking values in a product of finite sets $\Omega = \prod_{i=1}^d \Omega_i$,
    let $S \sub \Omega$, and let $\delta_1,\delta_2,\ldots,\delta_d>0$ with $\sum_{i=1}^d \delta_i \le \Pr[\vX\in S]$.
    There is a nonempty subset $T \sub S$ such that,
    for every $i \in [d]$ and $\vix \in T_{\le i-1}$, 
    $\Pr[\vX_{\leq i} \in T_{\le i} \mid \vX_{\leq i-1} = \vix] \ge \delta_i$.
\end{lemma}
\begin{proof}
    We use induction on $d$. As a base case, we note that if $d = 1$, we can just let $T = S$.
    For the induction step, 
    write $\Omega' = \prod_{i=1}^{d-1} \Omega_i$, and let $S' \sub \Omega'$ be given by
    \[S' = \big\{ \vix \in \Omega' \,\big\vert\, \Pr[\vX' = \vix] > 0 \,\wedge\, \Pr[\vX \in S \mid \vX' = \vix] \ge \delta_d \big\}. \]
    We first claim that $\Pr[\vX' \in S'] \ge \delta_1+\cdots+\delta_{d-1}$. 
    Indeed, this follows from the inequality
    \begin{align*}
     \Pr[\vX \in S] &= 
        \Pr[\vX \in S \,\wedge\, \vX' \in S'] + \Pr[\vX \in S \mid \vX' \notin S']\Pr[\vX' \notin S']\\
        &\le \Pr[\vX' \in S'] + \Pr[\vX \in S \mid \vX' \notin S'] ,
    \end{align*}
    which implies that
    \[\Pr[\vX \in S] - \Pr[\vX' \in S'] \le \Pr[\vX\in S \mid \vX' \notin S'] \le \delta_d .\]
    (Here we have assumed that $\Pr[\vX' \notin S'] > 0$; otherwise the claim is obvious.)
    
    Thus, we may apply the induction hypothesis on $\vX'$, $S'$, and $\delta_1,\ldots,\delta_{d-1}$, 
    to obtain a nonempty subset $T' \sub S'$ satisfying $\Pr[\vX_{\leq i} \in T_{\le i} \mid \vX_{\leq i-1} = \vix] \ge \delta_i$ for all $i \in [d-1]$ and $\vix \in T'_{\le i-1}$. Let $T \sub S$ be the nonempty set $\setmid{\vx \in S}{\vx' \in T'}$; we claim that $T$ is the desired set.

    To show this, we first note that $T'_{\le i} = T_{\le i}$ for all $0 \leq i \leq d-1$;
    indeed, since $T' \subseteq S'$, for every $\vy \in T'$ there is some $\vx \in T$ with $\vx' = \vy$. 
    It follows that $\Pr[\vX_{\leq i} \in T_{\le i} \mid \vX_{\leq i-1} = \vix] \ge \delta_i$ for every $i \in [d-1]$ and $\vix \in T_{\le i-1}$.
    It now suffices to show the condition for $i=d$. We have, for every $\vx \in T_{\le d-1}=T'$,
    \[
    \Pr[\vX_{\leq d} \in T_{\le d} \mid \vX_{\leq d-1} = \vx]
    = \Pr[\vX \in T \mid \vX' = \vx] \geq \delta_d,
    \]
    where the inequality uses the fact that $T' \sub S'$.
    This completes the proof.
\end{proof}

We next define (co-)lexicographic order for a random vector in $\NN^d$.
We will use it later in \labelcref{eq:Rp-def} to define our probabilistic analogue of local rank (where $Y_i$ below corresponds to the rank of a matrix slice at the point $(p_1,\ldots,p_i,X_{i+1},\ldots,X_d)$).

\begin{definition}[$\preceq_\delta$] \label{def:app-lex}
Let $\vX=(X_1,\ldots,X_d)$ be a random vector taking values in a product of finite sets $\Omega = \prod_{i=1}^d \Omega_i$,
and let $\vY = (Y_1,\ldots,Y_d)$ be a random vector taking values in $\NN^d$, where $Y_i$ is a function of $X_{i+1},X_{i+2},\ldots,X_d$ (in particular, $Y_d$ is always constant).\\
For a tuple $\vr \in \NN^d$ and a real $\delta > 0$, 
we write $\vY \preceq_\delta \vr$ if the following recursive definition holds:
\begin{itemize}
    \item If $d = 1$, then $\vY \preceq_\delta \vr$ if $Y_1 \leq r_1$.
    \item If $d > 1$, then $\vY \preceq_\delta \vr$ if either (1) $Y_d < r_d$, or (2) $Y_d = r_d$ and $\Pr_{X_d} [(\vY' \mid X_d) \preceq_\delta \vr'] \geq \delta$.\footnote{Here, $(\vY' \mid X_d)$ is the variable $\vY'$ conditioned on a fixed value of $X_d$. This condition makes sense since if we condition on a value of $X_d$, then $\vX'$ and $\vY'$ are random vectors satisfying the conditions of \cref{def:app-lex}.}
\end{itemize}
\end{definition}
\begin{remark}
The condition $\vY \preceq_\delta \vr$ is dependent on the choice of $\vX$.
However, in practice the choice of $\vX$ is clear and thus mention of $\vX$ is omitted from the notation.
\end{remark}

We use \cref{lem:pruning} to deduce that for a random tuple $\vY$, one can find a tuple that bounds it---in the sense of $\preceq_\d$---in any ``large enough'' set.

\begin{coro} \label{lemma:lexMarkov}
Let $d, \vX, \vY$ be as in \cref{def:app-lex} such that $d > 1$, and let $\eps > 0$. Suppose there is a subset $A \subseteq \NN^d$ such that $\Pr[\vY \in A] \geq \eps$. Then there is some $\vr \in A$ such that $\Pr[\vY = \vr] > 0$ and $\vY \preceq_{\eps/(d-1)} \vr$.
\end{coro}
\begin{proof}
Without loss of generality, replace $A$ with $\setmid{\vr \in A}{\Pr[\vY = \vr] > 0}$ so the condition $\Pr[\vY = \vr] > 0$ becomes automatic. For ease of notation, we reindex by letting $\vZ \coloneqq (X_2,X_3,\ldots,X_d)$ and $\Sigma_i = \Omega_{i+1}$.  For every $i\in[d]$, $Y_i$ depends on $\vZ_{\ge i}$, and for every $i\in[d-1]$, $Z_i$ takes values in $\Sigma_i$.

Put $\d=\eps/(d-1)$.
There is a set $T \sub \prod_{i=1}^{d-1} \Sigma_i$ such that 
$\vY(\vz) \in A$ for all $\vz \in T$, and 
\begin{equation}\label{eq:ind-T}
    \Pr[\vZ_{\geq i} \in T_{\ge i} \mid \vZ_{\geq i+1} = \viz] \geq \d
\end{equation}
for every $i \in [d-1]$ and $\viz \in T_{\ge i+1}$.
Indeed, this follows by applying \cref{lem:pruning} with $d-1$ as $d$, 
$(Z_{d-1},\ldots,Z_1)$ as $\vX$, $\setmid{(z_{d-1},\ldots,z_1)}{\vY(z_1,\ldots,z_{d-1}) \in A} \subseteq \prod_{i=d-1}^1 \Sigma_i$ as $S$, and $\d$ as all the $\d_i$,
using the fact that $\Pr[(Z_{d-1},\ldots,Z_1) \in S] = \Pr[\vY \in A] \ge \eps = \sum_{i=1}^{d-1} \d$.

Let $\vr = \maxprec_{\vz \in T} \vY(\vz)$,
and note that, by construction, $\vr \in A$.
We claim that $\vY \preceq_\d \vr$. In order to prove this, we prove via induction on $1 \leq i \leq d$ the following hypothesis:
\[\forall \viz \in T_{\ge i} \colon \vY_{\ge i+1}(\viz)=\vr_{\ge i+1} \implies \vY_{\le i}(-,\viz) \preceq_\delta \vr_{\le i} \;.\]
In other words, if $\viz$ satisfies the conditions above then $\vY(-,\viz) \preceq_\delta \vr$.
Note that the $i = d$ case of the hypothesis simply says $\vY \preceq_\delta \vr$, implying the desired statement.

The base case $i=1$ reads: for every $\vz \in T$, if $\vY_{\ge 2}(\vz)=\vr_{\ge 2}$ then $Y_1(\vz) \le r_1$ (as $Y_1(-,\vz)$ is constant);
it holds since $(Y_1(\vz), \vr_{\ge 2}) = \vY(\vz) \preceq \vr$, by construction.

For the inductive step, let $2 \leq i \leq d$ and assume the inductive hypothesis for $i-1$. Take some $\viz \in T_{\ge i}$ satisfying $\vY_{\ge i+1}(\viz)=\vr_{\ge i+1}$.
Put $y_i = Y_i(\viz)$,
so that $\vY_{\ge i}(\viz) = (y_i,\vr_{\ge i+1})$.
Note that $y_i \le r_i$, 
since otherwise,
for any $(\viz',\viz) \in T$---which exists---we have $\vY_{\ge i}(\viz',\viz) = \vY_{\ge i}(\viz) \succ \vr_{\ge i}$, contradicting the fact that $\vY(\viz',\viz) \preceq \vr$ by construction.

Now, if $y_i < r_i$ then, by \cref{def:app-lex}, $\vY_{\le i}(-,\viz) \preceq_\delta \vr_{\le i}$, as needed.
Therefore, suppose that $y_i=r_i$, so that $\vY_{\ge i}(\viz)=\vr_{\ge i}$.
By the induction hypothesis, for every $z \in \Sigma_{i-1}$ satisfying $(z,\viz) \in T_{\geq i-1}$ we have
$\vY_{\le i-1}(-,z,\viz) \preceq_\delta \vr_{\le i-1}$.
Thus, by \labelcref{eq:ind-T},
\[\Pr_{Z_{i-1}} [\vY_{\le i-1}(-,Z_{i-1},\viz) \preceq_\delta \vr_{\le i-1}]
\ge \Pr_{Z_{i-1}}[(Z_{i-1},\viz) \in T_{\ge i-1}] 
\ge \d.\]
By \cref{def:app-lex} again, 
we deduce that $\vY_{\le i}(-,\viz) \preceq_\delta \vr_{\le i}$.
This completes the induction step and, as explained above, the proof.
\end{proof}

\subsection{Probabilistic local rank} \label{subsec:problr}
Let $d \in \NN$, and let $T \in \cT^d$ and $\vp \in V^d$. Let $\vX^\vp$ be a $V^d$-valued random variable defined by letting, for all $i \in [d]$ repeatedly in descending order, $X^\vp_i$ be a uniformly random element of $\ker T[\vp \ominus_i \vX^\vp]_i$. This process makes sense since $T[\vp \ominus_i \vX^\vp]_i$ only depends on $X^\vp_j$ for $j > i$. 
Moreover, 
let $\vR^\vp = \vR^\vp(T)$
be the random variable where 
\begin{equation}\label{eq:Rp-def}
    R^\vp_i = \rank T[\vp \ominus_i \vX^\vp]_i.
\end{equation}

The random variables $\vX^\vp$ and $\vR^\vp$ satisfy some important properties.
By~\labelcref{eq:ext-ZT}, if $(\vP,\vX) \in \Z(T)\times \Z(T)$ is a uniform resample from $\Z(T) \sub V^d$ and $\vp \in \Z(T)$, then $\vX$ conditioned on $\vP = \vp$ is precisely $\vX^\vp$;\footnote{Indeed, $X^\vp_i$ is chosen uniformly at random from $\ker T[\vp \ominus_i \vX^\vp]_i=\ext_{\Z(T)}^i(\vp \ominus_{i} \vX^\vp)$.}
furthermore, we always have $\vp \adjoint{T} \vX^\vp$.\footnote{Indeed, for each $i=d,\ldots,2$, since $X^\vp_i \in \ker T[\vp \ominus_i \vX^\vp]_i$ we have $\vp \ominus_{i-1} \vX^\vp = (\vp \ominus_i \vX^\vp) \odot_i X^\vp_i \in \Z(T)$.}

We further note that $\vR^\vp$, together with $\vX^\vp$, satisfies that conditions in \cref{def:app-lex} 
(i.e., $R^\vp_i \in \NN$ depends on $X^\vp_{i+1},\ldots,X^\vp_d \in V$). 
We now proceed to deduce information about $T$ from the condition $\vR^\vp \preceq_\delta \vr$.
Recall the notation $T[v]=T(-,\ldots,-,v) \in \cT^{d-1}$.

\begin{claim}\label{claim:R-LR-bound}
    Let $\vr \in \NN^d$ and $\d > 0$.
    If $\vR^\vp \preceq_{\d} \vr$ then $\bLR_\vp(T) \preceq \vr$,
    provided $\d\abs{\F} > \prod_{i=1}^{d-1} (r_i+1)$.
\end{claim}
\begin{proof}
    We proceed by induction on $d$. The induction basis $d=1$ follows trivially since for any $T \in \cT^1$ and $\vp \in V$ we have $\vR^\vp(T) = (\rk T) = \bLR_\vp(T)$, regardless of $\vp$.
    For the induction step we will rely on the recursive nature of $\vR^\vp$, 
    namely that 
    \[\vR^\vp(T) = (\vR^{\vp'}(T[X^\vp_d]),\, \rk T[\vp]_d ).\footnote{This should be interpreted as saying that, conditioning $X_d^\vp = x$ for some $x \in V$, $(\vX^\vp)'$ has the same distribution as $\vX^{\vp'}$ applied to $T[x]$, and that $\vR^\vp(T)'$ and $\vR^{\vp'}(T[x])$ are the same functions of said variables.}\]

    Now, assume $\vR^\vp \preceq_{\d} \vr$, and let us prove that $\bLR_\vp(T) \preceq \vr$. \cref{def:app-lex} for $\vR^\vp(T) \preceq_{\d} \vr$ reads as follows:
    \begin{itemize}
        \item $R^\vp_d \le r_d$, and 
        \item $R^\vp_d = r_d$ implies 
        $\Pr_{X_d^\vp}[\vR^{\vp'}(T[X_d^\vp]) \preceq_{\delta} \vr'] \ge \delta$.
    \end{itemize}
    We may assume that the last coordinate of $\bLR_\vp(T)$, which is $\rk T[\vp]_d$, is equal to $r_d$;
    indeed, $\rk T[\vp]_d = R^\vp_d \le r_d$, and if $\rk T[\vp]_d < r_d$ then $\bLR_\vp(T) \prec \vr$ so we are done.

    Thus, we need to prove that $(\bLR_\vp(T))' \preceq \vr'$. Equivalently, by
    \cref{prop:lrdefrecur}, we need to prove  
    \begin{equation}\label{eq:large-field-goal}
        \maxprec_{x \in \bker T[\vec p]_d} \bLR_{\vp'}(T[x]) \preceq \vr' .    
    \end{equation}
    By the induction hypothesis applied to $T[x_d] \in \cT^{d-1}$ for each $x_d$ satisfying $\vR^{\vp'}(T[x_d]) \preceq_\delta \vr'$, we deduce that
    $\Pr_{X_d^\vp}\brac[\big]{\bLR_{\vp'}(T[X_d^\vp]) \preceq \vr'} \ge \delta$.
    By \cref{coro:LR-SZ} and our assumption on $\abs{\F}$, $\bLR_{\vp'}(T[x]) \preceq \vr'$ for every $x \in \bker T[\vp]_d$. 
    This proves~\labelcref{eq:large-field-goal}, thus completing the proof.
\end{proof}

\subsection{Proof of \texorpdfstring{\cref{main:LR-AR}}{Theorem \ref{main:LR-AR}}}\label{subsec:rank-vec}
Our goal is to prove the following statement.
\begin{quote}
For every $d \in \NN$ there is $C_d > 0$ such that the following holds:
Let $\eps \in (0,1]$. 
Every $T \in \cT^d$ has an LR-stable point $\vp \in \Z(P)$ satisfying $\norm{\LR_\vp(T)} \le (2^d -1+\eps)\AR(T)$, provided $\abs{\F}\geq C_d(1+\AR(T))^{d-1}/\eps$.
\end{quote}

If $d = 1$ the result is trivial; now assume otherwise.
Let $T \in \cT^d$. 
Let $(\vP,\vX) \in \Z(T)\times \Z(T)$ be a uniform resample from $\Z(T) \sub V^d$, and let
\[\vR = (\rank T[\vP \ominus_1 \vX],\ldots,\rank T[\vP \ominus_d \vX]).\]
By \cref{claim:uniform-resample} and \cref{prop:exp-rk-bound}, we have
\[\E[\norm{\vR}] \leq \norm{(\AR(T),\ldots,\AR(T))} = (2^d - 1) \AR(T).\]
Recall that for any $\vp \in Z(T)$, $\vX$ conditioned on $\vP = \vp$ is precisely $\vX^\vp$. 
Thus, there is some $\vp \in \Z(T)$ with $\E[\norm{\vR^\vp}] \leq (2^d-1)\AR(T)$, which we henceforth fix.

Let $\eps' \in (0,1)$ to be chosen later. 
Our goal next is to convert the above bound on the norm $\vR^\vp$ to a
bound on the vector $\vR^\vp$ itself.
By Markov's inequality, 
\[\Pr[\norm{\vR^\vp} \leq (2^d-1)\AR(T)/(1-\eps')] \geq \eps'.\]
Applying \cref{lemma:lexMarkov}, we get a ``lexicographic bound'' on $\vR^\vp$:
\[\vR^\vp \preceq_{\e'/(d-1)} \vr
\quad\text{ with }\quad \norm{\vr} \le (2^d -1)\AR(T)/(1-\e'),\, \vr \preceq \LR_\vp(T).\]
Applying \cref{claim:R-LR-bound}, we find that $\bLR_\vp(T) \preceq \vr$, provided 
\[
    \frac{\eps'}{d-1}\abs{\F} > \prod_{i=1}^{d-1} (r_i+1);
\]
    in this case, since $\vr \preceq \LR_\vp(T) \preceq \bLR_\vp(T) \preceq \vr$, we have $\LR_\vp(T) = \bLR_\vp(T)$, 
    or equivalently,
    \begin{equation}\label{eq:pre-AR-norm-bound}
        \norm{\LR_\vp(T)} \le (2^d -1)\AR(T)/(1-\e')
        \quad\text{ where $\vp$ is LR-stable for $T$.}
    \end{equation}

    It remains to bound $\abs{\F}$ independently of $\vr$. By the AM-GM inequality,
\begin{align*}
    (r_1 + 1) \cdots (r_{d-1} + 1) &= 
    \frac{(2^{d-1} r_1 + 2^{d-1}) \cdots (2r_{d-1} + 2)}{2^{d(d-1)/2}} \\
    &\leq \frac{1}{2^{d(d-1)/2}} \paren*{\frac{2^{d-1} r_1 + \cdots + 2r_{d-1} + 2^d - 2}{d-1}}^{d-1} \\
    &< \paren*{\frac{\Abs{\vec r} + 2^d - 1}{(d-1)2^{d/2}}}^{d-1} \\
    &\leq \paren*{\frac{(2^{d} - 1)(1 + \AR(T)/(1-\eps')))}{(d-1)2^{d/2}}}^{d-1} \\
    &< \frac{2^{d(d-1)/2}}{(d-1)^{d-1}} \frac{(1+\AR(T))^{d-1}}{(1-\eps')^{d-1}}.
\end{align*}
Thus, there is a constant $C'_d = \frac{2^{d(d-1)/2}}{(d-1)^d}$ such that if
\begin{equation} \label{eq:cpd} \abs{\F} \geq C'_d \frac{(1+\AR(T))^{d-1}}{\eps'(1-\eps')^{d-1}},\end{equation}
then \labelcref{eq:pre-AR-norm-bound} holds.
At this point, we choose
\[\eps' = \frac{\eps}{2^d - 1 + \eps},\]
so that $\frac{2^d-1}{1-\eps'} = 2^d-1+\eps$. Since $\frac{\eps}{2^d} \leq \eps' \leq \frac{1}{2}$, we have
\[\frac{1}{\eps'(1-\eps')^{d-1}} \le \frac{2^d}{\eps}\cdot\frac{1}{(1-\frac12)^{d-1}} = \frac{2^{2d-1}}{\eps}.\]
We conclude that \labelcref{eq:cpd} holds provided
\[\abs{\F} \geq 2^{2d-1} C'_d \frac{(1+\AR(T))^{d-1}}{\eps},\]
which completes the proof (with $C_d = 2^{2d-1} C'_d = 2^{O(d^2)}$).

\section{Partition Rank is Bounded by Local Rank} \label{sec:prlr}

Recall that \cref{cor:algrank} specifies a family of formulas $\cF_r$ such that for every matrix of rank $r$, there is a formula that gives a partition rank decomposition with $r$ terms.
The goal of this section is to generalize this result to higher-order tensors, 
by specifying a family $\cF_{\vp,\vr}$ of formulas such that for every tensor of local rank $\vr$ at an LR-stable point $\vp$, there is a formula that gives a partition rank decomposition with $\norm{\vr}$ terms.
Throughout this section, the reader may find it intuitive 
to think of elements of $\cT^d$ not as maps, but as higher-dimensional arrays of scalars.

Let us recall the notation $\rat{R}(T) \equiv T'$, where $\rat{R}=\ffrac{F}{g} \in \FRat(\cT^d,\cT^d)$ with $F \in \Poly(\cT^d,\cT^d)$ and $g \in \Poly(\cT^d)$, which means $F(T) = g(T)T'$, that is, $F(T) \in \cT^d$ and $T' \in \cT^d$ are equal up to a scalar (which may be $0$).

\begin{theo} \label{thm:lrprintermed}
    For every $\vp \in V^d$ and $\vr \in \NN^d$
    there is a set $\cF_{\vp,\vr}\subseteq \FRat(\cT^d, \cT^d)$  such that
        \begin{enumerate}
        \item\label{item:F-pr} $\PR(\rat{R}) \leq \norm{\vr}$ for all $\rat{R} \in \cF_{\vp,\vr}$;
        \item\label{item:F-equiv} If $T \in \overline \cT^d$ with $\bLR_\vp(T) \preceq \vr$, $\rat{R}(T) \equiv T$ for all $\rat{R} \in \cF_{\vp,\vr}$;
        \item\label{item:F-def} If $T \in \cT^d$ with $\LR_\vp(T) = \vr$, then there is some $\rat{R} \in \cF_{\vp,\vr}$ that is defined at $T$.
    \end{enumerate}
\end{theo}

\cref{thm:lrprintermed} is most informative for tensors with an LR-stable point, easily giving a proof of \cref{main:LR-PR}, which we briefly recall:
\begin{quote}
    If $\vp$ is an LR-stable point for a tensor $T$ then $\PR(T) \leq \Abs{\LR_\vp(T)}$.
\end{quote}
\begin{proof}[Proof of \cref{main:LR-PR}]
Let $\vp$ be an LR-stable point for a tensor $T \in \cT^d$.
Since $\LR_\vp(T) = \bLR_\vp(T)$, applying \cref{thm:lrprintermed} with $\vp$ and $\vr=\LR_\vp(T)$ implies that there exists $\rat{R} \in \FRat(\cT^d, \cT^d)$ with $\PR(\rat{R}) \le \Abs{\LR_\vp(T)}$ 
for which $\rat{R}(T) \equiv T$ and $\rat{R}$ is defined at $T$.
In other words, there is a polynomial map $F \in \Poly(\cT^d,\cT^d)$ and a polynomial $g \in \Poly(\cT^d)$ such that $\PR(F) \le \Abs{\LR_\vp(T)}$ and $F(T)=g(T)T$ with $g(T) \neq 0$.
Therefore, 
\[\PR(T) = \PR(g(T)T) = \PR(F(T)) \le \PR(F) \le \Abs{\LR_\vp(T)} ,\]
where the first inequality is by the definition of $\PR(F)$.
This completes the proof.
\end{proof}

\subsection{Derivatives}\label{subsec:derivatives}
Before we prove \cref{thm:lrprintermed}, we first need some technical preliminaries about derivatives of polynomial maps and, more generally, formal rational maps.
The derivative of a polynomial map $F \in \Poly(V,U)$ is a polynomial map $\Jac F \in \Poly(V,\, U \otimes V^*)$, where $V^*$ is the dual space of $V$ (specifying a direction for the derivative);
explicitly, it is given by a matrix
$(\partial_j F_i)_{i,j}$ where $F_i$ denote the coordinates of $F$ in some basis of $U$ and $j$ runs over a basis of $V$.
More generally, the derivative of a formal rational map $\ffrac{F}{g} \in \FRat(V,U)$ 
is defined in the obvious way;
\[\nabla \ffrac{F}{g} \coloneqq \ffrac{g\nabla F - F\otimes \nabla g}{g^2} \in \FRat(V,\, U\otimes V^*).\]
In the special case $U = \cT^d$, we naturally identify $U \otimes V^*$ with $\cT^{d+1}$, where the factor of $V^*$ goes to the ``last slot'' of the tensor.
Thus, if $F \in \Poly(V,\cT^d)$ then $\Jac F \in \Poly(V,\cT^{d+1})$.

We will also need a slightly more general definition of a derivative: 
given a polynomial map $F \in \Poly(W \times V,\, U)$, we can take partial derivatives with respect to only the variables in $V$ to yield $\nabla F \in \Poly(W \times V,\, U \otimes V^*)$. We can also define the derivative of an element of $\FRat(W \times V,\, U)$ to be an element of $\FRat(W \times V,\, U \otimes V^*)$ in a similar manner to the above.

We now list several easy properties of derivatives that will be useful later.
Recall the notation $T[v]=T(-,\ldots,-,v) \in \cT^{d-1}$.
\begin{claim}\label{claim:Jac-props}
    Let $\rat{R}=\ffrac{F}{g} \in \FRat(W\times V,\, U)$ and $H \in \Poly(W \times V,\, U)$. If all derivatives are taken with respect to $V$, we have
    \begin{enumerate}[label=(\roman*)]
        \item\label{item:J-const} If $c \in \Poly(W, U)$, then $\Jac (\ffrac{1}{c}\rat{R}) = \ffrac{c}{c^2}\Jac \rat{R}$
        \item\label{item:J-linear} $\Jac (\rat{R}+H) = \Jac \rat{R} + \Jac H$
        \item\label{item:J-double} If $U=\cT^d$ then $\PR(\Jac\rat{R}) \le 2\PR(\rat{R})$
        \item\label{item:J-equiv} If $\rat{R} \equiv H$ then $\Jac \rat{R} \equiv \Jac H$
        \item\label{item:J-T'} 
        For every $T \in \cT^d$ and every $y \in V$, we have $(\Jac T[-])(y) = T$
        \item\label{item:J-domain} $\Jac \rat{R}$ is defined at $x$ if and only if $\rat{R}$ is defined at $x$.
    \end{enumerate}
\end{claim}
\begin{proof} \mbox{}\par
    \begin{enumerate}[label=(\roman*)]
        \item We have 
        \[\Jac \Big(\ffrac{1}{c}\rat{R}\Big) = \Jac \ffrac{F}{cg} = \ffrac{1}{(cg)^2}(cg\Jac F - F\otimes\nabla (cg))
        = \ffrac{c}{c^2} \Jac \rat{R} .\]

        \item We have
        \begin{align*}
                \Jac (\rat{R}+H) &= \Jac \ffrac{F+gH}{g}
                = \ffrac{1}{g^2}\big( g\Jac (F+gH) - (F+gH)\otimes\nabla g \big)\\
                &= \ffrac{1}{g^2}\big( g\Jac F - F\otimes\nabla g + g^2\Jac H  \big)
                = \Jac \rat{R} + \Jac H.
        \end{align*}
        
        \item We have $\PR(\Jac \rat{R}) = \PR(R')$ with $R'=g\Jac F - F\otimes\nabla g$.
        Since $R'$ is linear in $F$, 
        it suffices to only prove the case where $\PR(\rat{R})=\PR(F)$ is $1$.
        Write $F = P \otimes Q$; then
        \begin{align*}
            \begin{split}
                R' &= g( \Jac P \otimes Q+ P \otimes \Jac Q) - P \otimes Q\otimes\nabla g
                = P \otimes Q' + P' \otimes Q
            \end{split}
        \end{align*}
        with $Q'\coloneqq g\Jac Q - Q\otimes \nabla g$ and $P'\coloneqq g \Jac P$, 
        which implies $\PR(\Jac \rat{R}) = \PR(R') \le 2$ as needed.
        
        \item If $\rat{R} \equiv H$, that is, $F = gH$, then 
        \[\Jac \rat{R} = \ffrac{1}{g^2}(g\Jac (gH) - gH\otimes\nabla g)
        = \ffrac{1}{g^2}(g^2\Jac H) \equiv \Jac H.\]
        
        \item By the linearity of $T[-] \colon V \to \cT^{d-1}$ (and our identification of $\cT^{d-1} \otimes V^*$ with $\cT^d$),
        its derivative $\Jac T[-] \in \Poly(V,\cT^d)$ is a constant map, mapping every input to $T$.

        \item
        $\Jac \rat{R}$ is defined at $x$ if and only if $g(x)^2 \neq 0$ if and only if $g(x) \neq 0$ if and only if $\rat{R}$ is defined at $x$, as needed. \qedhere
    \end{enumerate}
\end{proof}

Let us one last time recall the notation 
$T[\vx]_i=T(x_1,\ldots,x_{i-1},-,x_{i+1},\ldots,x_d) \in \cT^1$, 
and
$T[v]=T(-,\ldots,-,v) \in \cT^{d-1}$.

\subsection{Proof of \texorpdfstring{\cref{thm:lrprintermed}}{Theorem \ref{thm:lrprintermed}}}\label{subsec:lrpr-pf}
    We proceed by induction on $d$. The induction basis $d=1$ is precisely \cref{cor:algrank}, using the fact that for any $T \in \cT^1$ and $\vp \in V^1$, $\bLR_\vp(T)=\LR_\vp(T)=(\rk(T))$.
    
    For the induction step, assume that the result holds for $d-1$.
    Choose 
    \[\rat{R}' \in \cF_{\vp',\vr'},\, \rat{P} \in \cP_{r_d},\, \text{and } y \in V ,\]
    using the induction hypothesis for the first choice (and for the second choice recall \cref{def:Proj}).
    Now let $\rat{Q} = \Id - \rat{P}$, and write $\rat{P}=\ffrac{P}{g}$ and $\rat{Q}=\ffrac{Q}{g}$ for $P,Q \in \Poly(\cT^1,\cT^1)$ and $g \in \Poly(\cT^1)$.
    Recall \cref{def:tensor-proj}, which defines, for $M \in \Poly(\cT^1,\cT^1)$, a map $\Red^M_{\vp} \in \Poly(\cT^d \times V,\cT^{d-1})$  given by $\Red^M_{\vp}(T,v) = T[M(T[\vp]_d)v]$. Then, define the map $\rat{S} \in \FRat(\cT^d \times V, \cT^{d-1})$ given by
    \[\rat{S}(T, v) = \ffrac{1}{g(T[\vp]_d)}(\rat{R}' \circ \Red^P_\vp + \Red^Q_\vp).\]
    Recall $\nabla\rat{S} \in \FRat(\cT^d \times V, \cT^{d})$.
    Let $\rat{R} = \rat{R}_{\rat{R'},\rat{P},y} \in \FRat(\cT^d, \cT^d)$---our partition rank formula---be  
    \[\rat{R}(T) = \nabla \rat{S}(T, y)\;,\] 
    where the derivative $\nabla \rat{S} \in \FRat(\cT^d \times V, \cT^d)$ is with respect to $V$ only. We will show that the set
    \[\cF_{\vp,\vr} \coloneqq \setmid{\rat{R}_{\rat{R'},\rat{P},y}}{\rat{R'} \in \cF_{\vp',\vr'},\, \rat{P} \in \cP_{r_d},\, y \in V} ,\]
    satisfies the desired conditions.

    We first check Condition~\labelcref{item:F-pr}, that is, $\PR(\rat{R}) \leq \norm{\vr}$.
    By construction, $\PR(\rat{R'}) \leq \norm{\vr'}$ and $\PR(Q) \leq r_d$.
    Since $g(T[\vp]_d)$ does not depend on $v$, by parts~\labelcref{item:J-const} and~\labelcref{item:J-linear} of \cref{claim:Jac-props}, we find that
    \[\rat{R}(T) = \ffrac{g(T[\vp]_d)}{g(T[\vp]_d)^2} \paren*{\nabla(\rat{R}' \circ \Red^P_\vp)(T, y) + \nabla\Red^Q_\vp(T, y)}.\]
    Thus, to prove that $\PR(\rat{R}) \leq \norm{\vr}$ $(=2\norm{\vr'} + r_d)$, it suffices to show that $\PR(\nabla(\rat{R}' \circ \Red^P_\vp)) \leq 2\norm{\vr'}$ and that $\PR(\nabla\Red^Q_\vp) \leq r_d$. \cref{claim:Jac-props}\labelcref{item:J-double} shows that
    \[\PR(\Jac(\rat{R}' \circ \Red^P_\vp)) \le 2\PR(\rat{R}' \circ \Red^P_\vp) \le 2\PR(\rat{R'}) \leq 2\norm{\vr'}.\]
    Moreover, since $\Red^Q_\vp(T, v)$ is linear in $v$, we have that 
    $(\nabla \Red^Q_\vp(T,v))[u] = \Red^Q_\vp(T, u)$.
    Therefore,
    \[\PR(\nabla \Red^Q_\vp)
    \le \PR(Q(T[\vp]_d))
    \le \PR(Q) \leq r_d.\]
    
    Next, we prove Condition~\labelcref{item:F-equiv}, that is, $\rat{R}(T) \equiv T$ if $T \in \overline\cT^d$ satisfies $\bLR_\vp(T) \preceq \vr$.
    Let $x \in \overline V$; 
    we claim that $\rat{S}(T, x) \equiv T[x]$.
    Observe, crucially, that $T'_x:=\Red^P_\vp(T, x) \in \overline\cT^{d-1}$ 
    satisfies $\rat{R}'(T'_x) \equiv T'_x$, since $\bLR_\vp(T) \preceq \vr$ implies $\bLR_{\vp'}(T'_x) \preceq \vr'$ by \cref{lem:lralgrecur}.
    Therefore
    \[g(T[\vp]_d)T[x] = \Red^{P+Q}_\vp(T,x) = \Red^P_\vp(T, x) + \Red^Q_\vp(T, x) \equiv (\rat{R}' \circ \Red^P_\vp)(T, x) + \Red^Q_\vp(T, x).\]
    It follows that $\rat{S}(T, x) \equiv T[x]$, as claimed.
    Now, since $x$ is an arbitrary element of $\overline V$, we deduce the equality of polynomial maps $\rat{S}(T,-) \equiv T[-]$. 
    Taking derivatives yields that $\rat{R}(T) \equiv \nabla(T[-])(y) = T$,
    where we use \cref{claim:Jac-props}\labelcref{item:J-equiv}~and~\cref{claim:Jac-props}\labelcref{item:J-T'}.\footnote{This is the only time in the entire paper where we consider polynomial maps that do not necessarily have coefficients in $\F$. Thus, strictly speaking we need to apply the result over the field $\overline\F$.}
    
    Finally, we prove Condition~\labelcref{item:F-def}: suppose that $\LR_\vp(T) = \vr$, and let us prove that $\rat{R}$ is defined at $T$ for some choice of $\rat{R'},\rat{P},y$.
    Note that $\rat{R}$ $(=\nabla \rat{S}(T, y))$ is defined at $T$ if and only if $\rat{S}$ is defined at $(T,y)$, by \cref{claim:Jac-props}\labelcref{item:J-domain}. 
    By construction, this is equivalent to 
    $g(T[\vp]_d) \neq 0$ and $\rat{R}' \circ \Red^P_\vp$ being defined at $(T,y)$. 
    Thus, to prove that $\rat{R}$ is defined at $T$, it suffices to show that $\rat{P}$ is defined at $T[\vp]_d$, and that $\rat{R}'$ is defined at $\Red^P_\vp(T, y)$. 
    We now choose $\rat{R}' \in \cF_{\vp',\vr'}$, $\rat{P} \in \cP_{r_d}$, and $y \in V$
    as follows.
    Since $T[\vp]_d$ has rank $r_d$, by \cref{cor:algproj},
    there exists $\rat{P} \in \cP_{r_d}$ defined at $T[\vp]_d$. 
    By \cref{prop:lrdefrecur}, there exists some $x \in\ker T[\vp]_d$ such that $\LR_{\vp'}(T[x]) = \vr'$; since \cref{cor:algproj} guarantees that $P(T[\vp]_d)$ is a surjection onto $\ker T[\vp]_d$, we can thus find a $y \in V$ such that $\LR_{\vp'}(\Red^P_\vp(T)) = \vr'$. 
    Finally, we may find some $\rat{R}'\in \cF_{\vp',\vr'}$ defined at $\Red^P_\vp(T)$ by the inductive hypothesis. This completes the proof.

\section{Proofs of Main Results}\label{sec:mainproof}
For a tensor $T$ over a field $\F$, 
and a field extension $\K/\F$, 
we denote by $T^\K$ the corresponding tensor over $\K$.
We will need the following two bounds.
\begin{lemma}[{\cite[Proof of Corollary 1]{CohenMo21B}}]\label{lemma:PR-ext}
    For any finite field extension $\K/\F$ of degree $\ell$, 
    and any tensor $T$ over $\F$,
    \[\PR(T) \le \ell \PR(T^\K).\]
\end{lemma}

\begin{lemma}[\cite{ChenYe24,UniformStability}]\label{lemma:AR-ext}
    For any finite field extension $\K/\F$,
    and any tensor $T$ over $\F$,
    \[\AR(T^\K) = \Theta_k(\AR(T)).\]
\end{lemma}

\subsection{Proof of \texorpdfstring{\cref{main:main}}{Theorem \ref{main:main}}}
Let $C_k$ and $E_k$ be the constants from \cref{theo:main1} and the upper bound in \cref{lemma:AR-ext}, respectively;
without loss of generality, assume $C_k, E_k \geq 1$.
Let $\K/\F$ be a finite field extension of degree $\ell$ to be determined later.
Applying \cref{theo:main1} with $\e = 1$ on $T^\K$ yields
$\PR(T^\K) \leq 2^{k-1}\AR(T^\K)$, 
if $\abs{\K} \ge  C_k(\AR(T^\K)+1)^{k-2}$.
Thus, if we define $K_0 = C_k(E_k\AR(T)+1)^{k-2} \geq 1$ and let $\ell \geq 1$ be minimal such that $\abs{\F}^\ell \geq K_0$, applying \cref{lemma:PR-ext,lemma:AR-ext} yields
\[\PR(T) \le \ell \PR(T^\K) \le \ell 2^{k-1}\AR(T^\K) \le \ell 2^{k-1}E_k \AR(T) = O_k(\ell\AR(T)).\]
It suffices to show that $\ell = O_k(\lf_\F(\AR(T)))$. Indeed,
\[\ell \leq \log_{\abs{\F}}(K_0) + 1 \leq \log_{\abs{\F}}(C_kE_k^{k-2}(\AR(T)+1)^{k-2}) + 1 = O_k(1 + \log_{\abs{\F}}(\AR(T) + 1)) = O_k(\lf_\F(\AR(T))),\]
finishing the proof.

\begin{remark}
    It is not hard to work out the dependence on $k$ here;
    using $\log C_k=k^{O(1)}$ and \cite{ChenYe24,UniformStability}, we get $\ell \le k^{O(1)} \lf_\F(\AR(T))$,
    so $\PR(T) \le D_k \AR(T) \lf_\F(\AR(T))$ with 
    $D_k \le k^{O(1)} \cdot 2^k E_k = \exp(O(k\log\log k))$.
\end{remark}

\subsection{Polarization}
To extend results regarding $k$-tensors to polynomials of degree $k$, we describe the standard technique of polarization.
For this subsection, let $k \geq 2$ be a positive integer, $\F$ be a finite field of characterisitic greater than $k$, $\chi$ be a nontrivial additive character on $\F$, and $P$ be a polynomial over $\F$ of degree $k$ in $n$ variables. Recall that the rank of $P$, which we denote $\rk(P)$, is defined as the minimum $r$ such that the degree-$k$ part of $P$, denoted $P_k$, is the sum of $r$ reducible homogeneous polynomials.
\begin{definition}
The \emph{polarization} of $P$ is the polynomial $\polar{P} \colon (\F^n)^k \to \F$ given by
\[\polar{P}(x_1,\ldots,x_k) = \Delta_{x_1}\cdots\Delta_{x_k} P(y)
    = \sum_{S \sub [k]} (-1)^{d-\abs{S}} P\paren*{y + \sum_{i \in S} x_i},\]
where $\Delta_d f(x)$ denotes the finite difference $f(x+d) - f(x)$ and $y \in \F^n$ is arbitrary.
\end{definition}
As each finite difference lowers the degree by $1$, this is a multilinear form in the $x_i$ and independent of $y$ (see~\cite[Lemma 2.4]{GowersWo11}) and thus can be treated as a well-defined $k$-tensor. Moreover, $\polar{P}$ satisfies two important properties:
\begin{itemize}
\item We have
\[\abs{\E_x[\chi(P(x))]} \leq \abs{\E_{x_1,\dots,x_k}[\chi(\polar{P}(x_1,\ldots,x_k))]}^{1/2^{k-1}} = \abs{\F}^{-\AR(\polar{P})/2^{k-1}}.\]
This follows from repeated application of the Cauchy-Schwarz inequality (see e.g.~\cite[Lemma~3.1]{Schmidt84} or~\cite[Lemma~3.4]{GowersWo11}).
\item We have $P_k(x) = \frac{1}{k!} \polar{P}(x,x,\ldots,x)$, so $\rank(P) \leq \PR(\polar{P})$. (This is where we use the fact that the characteristic of $\F$ is greater than $k$.) In fact, from the definition of $\polar{P}$ we also have 
$\PR(\polar{P}) \leq 2^k \rank(P)$ (see e.g.~\cite[Claim~3.2]{LampertZi24}).
\end{itemize}

Now, applying \cref{main:main} to $P^\circ$ yields the following.
\begin{coro}\label{coro:general-poly}
If $\rk(P) \ge r$ then $\abs{\E_x[\chi(P(x))]} \le \abs{\F}^{-\Omega_k(r/\lf_\F(r))}$. 
\end{coro}

\subsection{Proof of \texorpdfstring{\cref{mcor:polys-equidistribution}}{Corollary \ref{mcor:polys-equidistribution}}}
Let $\chi$ be a nontrivial additive character of $\F$. By Fourier inversion, we have
\[\Pr_x[P(x) = 0] = \frac{1}{\abs{\F}^m} \sum_{a \in \F^m} \E_x[\chi(a \cdot P(x))].\]
The case $a=0$ contributes $\abs{\F}^{-m}$. 
It suffices to show that $\abs{\E_x[\chi(a \cdot P(x))]} \leq \abs{\F}^{-\Omega_k(r/\lf_\F(r))}$ for all $a \neq 0$. 
If $\deg(a \cdot P) \le 1$ then there is nothing to prove:
if $a \cdot P$ is constant then $r=0$ so the right hand side is $1$, and if $\deg(a \cdot P)=1$ then the left hand side is $0$.
Otherwise, we can apply \cref{coro:general-poly} on $a \cdot P$, since $\rk(a \cdot P) \geq r$, and we are done.

\begin{remark}\label{remark:Schmidt-subsume}
As mentioned in the introduction, a bound of 
$\abs{\F}^{-\Omega_k(r/\lf_\F(r))}$ 
is in fact stronger than Schmidt's bound of $O_{n,k}(\abs{\F}^{-\Omega_k(r)})$. To see this, we show 
$
    \abs{\F}^{-c_k r/\lf_\F(r)} \leq n^{c_k n/2} \abs{\F}^{-c_k r/2},
$
that is,
\[\abs{\F}^{-2r/\lf_\F(r)} \leq n^{n} \abs{\F}^{-r}.\]
Indeed, if $\abs{\F} \leq n$, then since $r \leq n$ the right-hand side
is at least $1$. Otherwise, $\lf_\F(r) \leq \lf_\F(n) \leq 2$, so the left-hand side is at most 
$\abs{\F}^{-r}$.
\end{remark}

\subsection{Proof of \texorpdfstring{\cref{mcor:barpr}}{Corollary \ref{mcor:barpr}}}
Let $T$ be a $k$-tensor over $\F$. Pick some finite extension field $\K$ such that $\PR(T^\K) = \PR(T^{\overline{\F}})$. Then, we have $\AR(T) \leq O_k(\AR(T^\K)) \leq O_k(\PR(T^\K)) \leq O_k(\PR(T^{\overline{\F}}))$,
as
the analytic rank is at most the partition rank. Thus, \cref{main:main} gives $\PR(T) \leq O_k(\PR(T^{\overline{\F}}) \lf_\F(\PR(T^{\overline{\F}}))$.
If $Q$ is multilinear, we are done as the notions of rank and partition rank coincide. Otherwise, let $T = \polar{Q}$, and as mentioned above, $\PR(T) = \Theta_k(\rank(Q))$ and $\PR(T^{\overline \F}) = \Theta_k(\brk(Q))$.

\subsection{Proof of \texorpdfstring{\cref{mcor:pgi}}{Corollary \ref{mcor:pgi}}}
Let $f = \omega^{P(x)}$ where $\omega = e^{2\pi i/p}$. Essentially by definition,
\[\Abs{\omega^{P(x)}}_{U^k} = \E_{x_1,\ldots,x_k}[\omega^{\polar{P}(x_1,\ldots,x_d)}] = p^{-\AR(\polar{P})},\]
so $\AR(\polar{P}) \leq \log_p(1/\delta)$. As a result \cref{main:main} shows that $\rank(P) \leq O_k(\log_p(1/\delta) \lf_p(\log_p(1/\delta)))$. To finish, we will show that for all $P$ of degree $k$, we have $\Abs{\omega^{P(x)}}_{u^k} \geq p^{-2\rank(P)}$.

To see why this is true, observe that $\omega^{-P_k(x)}$ can be written as a function of $Q_1,\ldots,Q_s$ where $Q_1,\ldots,Q_s$ are polynomials of degree less than $k$ and $s = 2\rank(P)$. Thus, by a Fourier transform we may write
\[\omega^{-P_k(x)} = \sum_{0\leq j_1,\ldots,j_s<p} c_{j_1,\ldots,j_s} \omega^{j_1 Q_1 + \cdots j_s Q_s}\]
where $\abs{c_{j_1,\ldots,j_s}} \leq 1$ as $\abs{\omega^{-P_k(x)}} \leq 1$. Therefore
\begin{multline*}
1 = \abs[\big]{\E_x [\omega^{P_k(x)} \omega^{-P_k(x)}]} = \abs*{\sum_{0\leq j_1,\ldots,j_s<p} c_{j_1,\ldots,j_s} \E_x[\omega^{P_k(x)} 
 \omega^{ j_1 Q_1 + \cdots j_s Q_s}]} \\ \leq \sum_{0\leq j_1,\ldots,j_s<p} \abs{c_{j_1,\ldots,j_s}} \abs[\big]{\E_x[\omega^{P_k(x)} 
 \omega^{ j_1 Q_1 + \cdots j_s Q_s}]} \leq p^s \Abs{\omega^{P_k(x)}}_{u^k} = p^s \Abs{\omega^{P(x)}}_{u^k}.
\end{multline*}
The result follows.

\printbibliography

\end{document}